\documentclass{article}

\usepackage{amssymb}
\usepackage{pifont}
\usepackage{amsmath,amsthm,amssymb,amsfonts}

\usepackage{amsfonts}
\usepackage{amsmath, amsthm}
\usepackage{amssymb, amsxtra, bm}
\usepackage{mathtools}
\usepackage{graphicx}

\usepackage{amscd}
\usepackage{multicol}
\usepackage{epstopdf}
\usepackage{color}

\usepackage[numbers,sort&compress]{natbib}
\usepackage{epsf}
\usepackage{float}
\usepackage{extarrows}
\usepackage[a4paper, text={.8\paperwidth,.83\paperheight}, ratio=1:1]{geometry}
\usepackage[format=hang,font=small,textfont=it]{caption}
\usepackage{authblk}

\usepackage[english]{babel}
\usepackage{latexsym}
\usepackage{nicefrac}

\usepackage{pgf, tikz, pgfplots}
\pgfplotsset{compat=1.15}
\usepackage{mathrsfs}
\usetikzlibrary{arrows}
\usetikzlibrary{calc}
\usepackage[tight]{subfigure}
\usepackage{hyperref}

\usepackage{color}
\definecolor{greenbean}{RGB}{199,237,204}

\usepackage[T1]{fontenc}%
\usepackage[utf8]{inputenc}%
\usepackage{lmodern}%
\usepackage{textcomp}%
\usepackage{lastpage}%
\usepackage{tikz}%
\usepackage{tikz-cd}%
\usepackage{float}
\usepackage{amsfonts, amsmath}
\usepackage{hyperref}
\usepackage{environ}
\usetikzlibrary{arrows}

\makeatletter
\newsavebox{\measure@tikzpicture}
\NewEnviron{scaletikzpicturetowidth}[1]{%
	\def\tikz@width{#1}%
	\begin{lrbox}{\measure@tikzpicture}%
		\BODY
	\end{lrbox}%
	\pgfmathparse{#1/\wd\measure@tikzpicture}%
	\BODY
}
\makeatother
\newtheorem{thm}{Theorem}[section]
\newtheorem{lemma}[thm]{Lemma}
\newtheorem{prop}[thm]{Proposition}
\newtheorem{cor}[thm]{Corollary}

\newtheorem{notation}{Notation}[]
\theoremstyle{definition}
\newtheorem{Def}[thm]{Definition}

\newtheorem{rmk}[thm]{Remark}

\newcommand{\CC}{\mathbb{C}}
\newcommand{\ZZ}{\mathbb{Z}}
\newcommand{\PP}{\mathbb{P}}
\newcommand{\mcC}{\mathcal{C}}
\newcommand{\mcB}{\mathcal{B}}

\newcommand{\mcO}{\mathcal{O}}
\newcommand{\mcQ}{\mathcal{Q}}

\newcommand{\Kbar}{\overline{K}}
\newcommand{\fd}{\mathfrak{d}}

\newcommand{\sr}{\mathrm{sr}}
\newcommand{\red}{\mathrm{r}}
\newcommand{\I}{\mathop{\mathrm{I}}\nolimits}
\newcommand{\Comb}{\mathop{\mathrm{Comb}}\nolimits}
\newcommand{\ord}{\mathop{\mathrm{ord}}\nolimits}

\newcommand{\Supp}{\mathop{\mathrm{Supp}}\nolimits}

\newcommand{\Div}{\mathop{\mathrm{Div}}\nolimits}
\newcommand{\pic}{\mathop{\mathrm{Pic}}\nolimits}

\newcommand{\MW}{\mathop{\mathrm{MW}}\nolimits}
\newcommand{\Contr}{\mathop{\mathrm{Contr}}\nolimits}

\def\cpt{\mathbb{CP}^2}
\newcommand{\pcpt}[1]{{\pi_1(\mathbb{CP}^2 \setminus #1, *)}}

\DeclarePairedDelimiterX\setc[2]{\{}{\}}{\,#1 \;\delimsize\vert\; #2\,}

\newcommand{\ug}[1]{\Gamma_#1}

\DeclareMathOperator{\pr}{pr}

\makeatletter
\newcommand{\uGammaSq}[2]{\begin{equation}\label{#1}
		\ug{#2}^2\uGammaSqChecknextarg}
	\newcommand{\uGammaSqChecknextarg}{\@ifnextchar\bgroup{\uGammaSqGobblenextarg}{ = e,
\end{equation} }}
\newcommand{\uGammaSqGobblenextarg}[1]{ = \ug{#1}^2\@ifnextchar\bgroup{\uGammaSqGobblenextarg}{  = e,
\end{equation}}}
\makeatother

\newcommand{\ubegineq}[1]{\begin{equation}\label{#1}}
\newcommand{\uendeq}{\end{equation}}

\makeatletter
\newcommand{\uProjRel}[2]{
\begin{equation}\label{#1}
	\ug{#2'}\ug{#2}\uProjRelChecknextarg}
\newcommand{\uProjRelChecknextarg}{\@ifnextchar\bgroup{\uProjRelGobblenextarg}{ = e
\end{equation}}}
\newcommand{\uProjRelGobblenextarg}[1]{\ug{#1'}\ug{#1}\@ifnextchar\bgroup{\uProjRelGobblenextarg}{  = e. \end{equation}}}
\makeatother

\newcommand{\mbb}[1]{\mathbb{#1}}

\newcommand{\notinsubfile}[1]{}

\begin{document}

\title{
The realization space of a certain conic line arrangement of degree 7 and a $\pi_1$-equivalent Zariski pair
\footnotetext{Email addresses.\\ Meirav Amram: meiravt@sce.ac.il; Shinzo Bannai: bannai@ous.ac.jp;  Taketo Shirane: shirane@tokushima-u.ac.jp; Uriel Sinichkin: sinichkin@mail.tau.ac.il;  Hiro-o Tokunaga: tokunaga@tmu.ac.jp\\

 2020 Mathematics Subject Classification. 14H30, 14H50, 14Q05, 20F55. \\{\bf Key words}: Zariski pairs, conic-line arrangements, fundamental group, Coxeter groups.}}
\author[1]{Meirav Amram}
\author[2]{Shinzo Bannai}
\author[3]{Taketo Shirane}
\author[4]{Uriel Sinichkin}
\author[5]{Hiro-o Tokunaga}

\affil[1]{\small{Department of Mathematics, Shamoon College of Engineering, Ashdod, Israel; ORCID  ID: 0000-0003-4912-4672}}
\affil[2]{\small{Department of Applied Mathematics, Faculty of Science, Okayama University of Science, Japan}}
\affil[3]{\small{Department of Mathematical Science, Faculty of Science and Technology, Tokushima University, Japan}}
\affil[4]{\small{School of Mathematical Sciences, Tel Aviv University, Tel Aviv, Israel}}
\affil[5]{\small{Department of Mathematical Sciences, Tokyo Metropolitan University, Japan}}

\maketitle

\abstract{In this paper, we continue the study of the embedded topology of plane algebraic curves. We study the realization space of conic line arrangements of degree $7$ with certain fixed combinatorics and determine the number of connected components. This is done by showing the existence of a Zariski pair having these combinatorics, which we identified as a $\pi_1$-equivalent Zariski pair.  


\section{Introduction}

In this paper, we continue the study of the embedded topology of plane (algebraic) curves.
Here the {\it embedded topology} of a plane curve $\mcC\subset\PP^2$ is the homeomorphism class of the pair $(\PP^2,\mcC)$ of the complex projective plane $\PP^2$ and the reduced divisor $\mcC$ on $\PP^2$.
It is known that the combinatorial type of $\mcC$  determines the embedded topology of $\mcC$ in its tubular neighborhood, but does not determine the embedded topology in $\PP^2$.   
Here, the  {\it combinatorial type} ({\it combinatorics} for short) of a plane curve, is the  data determined by the number of irreducible components, the degrees and singularities of components, and the intersections of components of $\mcC$.
A pair $(\mcC_1,\mcC_2)$ of plane curves $\mcC_1,\mcC_2\subset\PP^2$ is called a {\it Zariski pair} if $\mcC_1 $ and $\mcC_2$ have the same combinatorics but different embedded topology in $\PP^2$ (See \cite{bartolo94, survey} for a precise definition of Zariski pairs.).  The study of Zariski pairs can be divided into  the following two main steps:
\begin{enumerate}
\item Construct and study curves having a given fixed combinatorial type. 
\item Find an appropriate invariant to distinguish the embedded topology of the curves.
\end{enumerate}
In this paper, we address both steps for certain conic-line arrangements.

Concerning the first step, the second and fifth named authors, with their collaborators, used rational elliptic surfaces (for example: \cite{tokunaga2014}, \cite{bannai-tokunaga2015}, \cite{bannai2016}, \cite{bannai_tokunaga2019zariski})  to construct Zariski pairs of reducible curves whose irreducible components have small degree. 
Recently, their method was simplified by applying Mumford representations and Gr\"obner bases \cite{takahashi_tokunaga2020explicit}, \cite{takahashi_tokunaga2021}.
The above works were mostly restricted to constructing a few examples of curves with the given combinatorics. However, in this paper, we apply their methods in a more general form to conic-line arrangements of a specific fixed combinatorial type, then study the realization space of such arrangements. Here, the realization space of curves of degree $d$ with fixed combinatorics means the quasi-projective variety in $\PP {\rm H}^0(\PP^2,\mathcal{O}(d))$, consisting of closed points corresponding to such curves.

For the second step, the fundamental group $\pi_1(\PP^2\setminus\mcC)$ of the complement of a plane curve $\mcC\subset\PP^2$ has been  used to study the embedded topology of plane curves, from the initial work of Zariski \cite{zariski29}.
However, there exist Zariski pairs $(\mcC_1,\mcC_2)$ with $\pi_1(\PP^2\setminus\mcC_1)\cong\pi_1(\PP^2\setminus\mcC_2)$, which are called {\it $\pi_1$-equivalent Zariski pairs}.
For example, there are $\pi_1$-equivalent Zariski pairs of sextics with simple singularities in the list of \cite{shimada2009} (see \cite[Remark~5.9]{artal-carmona-cogolludo2002}). 
The plane curves consisting of one smooth cubic and one smooth curve of degree $d\geq 4$ studied by Shimada in \cite{shimada2003} were shown to be $\pi_1$-equivalent by the third named author in \cite{shirane2017}. 
Artal arrangements consisting of one smooth curve of degree $d\geq 4$ and three non-concurrent lines also produce many $\pi_1$-equivalent Zariski pairs (see \cite{artal-cogolludo-martin2020} and \cite{shirane2019}).
The above known $\pi_1$-equivalent Zariski pairs are given by curves containing a component with either singularities or genus $\geq 1$.
One of the goals of this paper is to give a $\pi_1$-equivalent Zariski pair of conic-line arrangements (i.e., plane curves consisting of only smooth rational curves). The fundamental groups of the conic-line arrangements of degree $7$ and $8$ given in \cite{tokunaga2014}, and an additional new example, have been calculated in \cite{robert}, but these were not $\pi_1$-equivalent. Also, for the conic-line arrangements distinguished by the existence or non-existence of certain dihedral covers such as given in \cite{bannai-tokunaga2015},  the fundamental groups  were not completely calculated but the existence or non-existence imply that they are not $\pi_1$-equivalent. One method to distinguish the topology of reducible curves $\mcC$, which works also for some $\pi_1$-equivalent cases, is to consider irreducible components $C\subset \mcC$ and the torsion classes of $\pic^0(C)$ derived from $\mcC$ (cf. \cite{shimada2003}, \cite{shirane2019}).
However, because $\pic^0(\PP^1)=0$, this  approach cannot be taken to give a Zariski pair for the  conic-line arrangement case. 
Hence, an example of a $\pi_1$-equivalent Zariski pair of conic-line arrangements is of interest.

The conic-line arrangements that we consider in this paper will have the following combinatorial type, denoted by $\Comb(\mcC)$, as depicted in Figure \ref{fig:comb}:  
\begin{enumerate}
    \item The arrangement consists of three smooth conics $C_1, C_2, C_3$ and a line $L$.
    \item $C_1$ and $C_2$ intersect transversally at $4$ distinct points $\{p_1, \ldots, p_4\}$.
    \item $C_3$ passes through two of the points $\{p_1, \ldots, p_4\}$ and is tangent to both $C_1$ and $C_2$ at points distinct from $\{p_1, \ldots, p_4\}$. We call such $C_3$ a {\it weak contact conic} of $C_1+C_2$. 
    \item $L$ is a bitangent line of $C_1+C_2$ (i.e., a line tangent to both $C_1$ and $C_2$),  and intersects $C_3$ transversally.  
\end{enumerate}
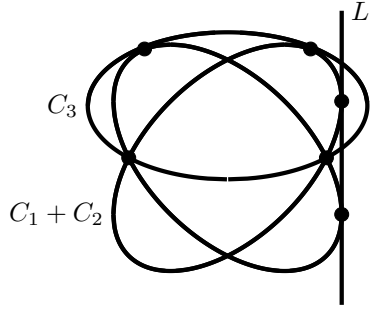
\begin{figure}
\centering
\begin{tikzpicture}[scale=1.3]
\draw[ultra thick] plot [domain=-2/sqrt(3):2*sqrt(3)/3, samples=200] (\x,{(\x+sqrt(-3*\x*\x+4))/2}); 
\draw[ultra thick] plot [domain=-2/sqrt(3):2*sqrt(3)/3, samples=200] (\x,{(\x-sqrt(-3*\x*\x+4))/2});
\draw[ultra thick] plot [domain=2/sqrt(3):-2*sqrt(3)/3, samples=200] ({(\x-sqrt(-3*\x*\x+4))/2},\x); 
\draw[ultra thick] plot [domain=2/sqrt(3):-2*sqrt(3)/3, samples=200] ({(\x+sqrt(-3*\x*\x+4))/2},\x);  
\draw[ultra thick] plot [domain=-2/sqrt(3):2*sqrt(3)/3, samples=200] (\x,{(-\x+sqrt(-3*\x*\x+4))/2}); 
\draw[ultra thick] plot [domain=-2/sqrt(3):2*sqrt(3)/3, samples=200] (\x,{(-\x-sqrt(-3*\x*\x+4))/2});
\draw[ultra thick] plot [domain=2/sqrt(3):-2*sqrt(3)/3, samples=200] ({(-\x-sqrt(-3*\x*\x+4))/2},\x); 
\draw[ultra thick] plot [domain=2/sqrt(3):-2*sqrt(3)/3, samples=200] ({(-\x+sqrt(-3*\x*\x+4))/2},\x);  
\draw[ultra thick] plot [domain=-1.5:1.5] ({2/sqrt(3)}, \x);
\draw[ultra thick] plot [domain=-sqrt(2):sqrt(2), samples=200] (\x, {3*sqrt(2)/8+3*sqrt(-2*\x*\x+4)/8});
\draw[ultra thick] plot [domain={(3*sqrt(2)-6)/8}:{(3*sqrt(2)+6)/8}, samples=200] ({sqrt(24*sqrt(2)*\x-32*\x*\x+9)/3}, \x);
\draw[ultra thick] plot [domain={(3*sqrt(2)-6)/8}:{(3*sqrt(2)+6)/8}, samples=200] ({-sqrt(24*sqrt(2)*\x-32*\x*\x+9)/3}, \x);
\draw[fill] (1,0) circle(2pt);
\draw[fill] (-1,0) circle(2pt);
\draw[fill] ({2/sqrt(3)}, {1/sqrt(3)}) circle(2pt);
\draw[fill] ({2/sqrt(3)}, {-1/sqrt(3)}) circle(2pt);
\draw[fill] ({0.84}, {1.11}) circle(2pt);
\draw[fill] ({-0.84}, {1.11}) circle(2pt);
\draw ({2/sqrt(3)}, 1.5) node [right] {$L$};
\draw (-{2/sqrt(3)}, {-1/sqrt(3)}) node [left] {$C_1+C_2$};
\draw ({-sqrt(2)}, {3*sqrt(2)/8}) node [left] {$C_3$};
\end{tikzpicture}
\caption{The combinatorial type $\Comb(\mcC)$.}
\label{fig:comb}
\end{figure}
 Let $[T, X, Z]$ be homogeneous coordinates of $\PP^2$ and let $t=\frac{T}{Z}$ and $x=\frac{X}{Z}$ be affine coordinates. We consider the following curves that give  realizations of $\Comb(\mcC)$:
\begin{align*}
C_1&: x-t^2=0, \\
C_2&: x^2-10tx+25x-36=0, \\
C_3&: x-\left(\frac{5}{4}t^2-2t+3\right)=0 && \mbox{(weak contact conic of $\mcQ:=C_1+C_2$)}, \\
L_1&: x-\left(\frac{32}{5}t-\frac{256}{25}\right)=0 && \mbox{(bitangent line of $\mcQ$)}, \\
L_2&: x=0 && \mbox{(bitangent line of $\mcQ$)}, \\
L_3&: x-(10t-25)=0 && \mbox{(bitangent line of $\mcQ$)}, \\
L_4&: x-\left(\frac{18}{5}t-\frac{81}{25}\right)=0  && \mbox{(bitangent line of $\mcQ$)},
\end{align*}
Then the four conic-line arrangements $\mcC_i$ of degree $7$ 
of the form
\begin{align*}
    \mcC_i&:=C_1+C_2+C_3+L_i \qquad (i=1,2,3,4)
\end{align*}
have the combinatorial type $\Comb(\mcC)$.
Our main result regarding these curves is as follows:
\begin{thm}\label{thm:main}
$(\mcC_i,\mcC_j)$ is a $\pi_1$-equivalent Zariski pair if $\{i,j\}\ne\{1,2\}, \{3,4\}$.  Furthermore, the realization space of conic-line arrangements having the combinatorial type ${\rm Comb}(\mcC)$ has exactly two connected components.
\end{thm}

We use the invariant called the splitting type, defined in \cite{bannai2016}, to distinguish the embedded topology of the conic-line arrangements in Theorem~\ref{thm:main}.
We can compute the splitting type for the curves in Theorem~\ref{thm:main} from the construction through the rational elliptic surfaces without using the defining equations.
We also demonstrate another computation of the splitting types, using the defining equations and Gr\"obner bases of $0$-dimensional ideals as verification.
By using Gr\"obner bases of $0$-dimensional ideals, the computation of the splitting types is simpler.

This paper is organized as follows. In Section \ref{sec:construction} we briefly describe Mumford representations and their relation with Gr\"obner bases. Then we describe the construction of the curves, which utilizes the theory of elliptic surfaces. In Section \ref{sec:deformation}, we study the realization space of conic-line arrangements having  combinatorics $\Comb(\mcC)$. In Section \ref{sec:splitting} we calculate the splitting type in two ways - one uses the theory of elliptic surfaces and is conceptual, while the other is a concrete calculation using Gr\"obner bases - both prove the existence of a Zariski pair with $\Comb(\mcC)$. 
The existence of a Zariski pair makes it possible to  determine the number of connected components. 
In Section \ref{sec:fundamental-groups} we calculate the fundamental groups of the curves in each connected component. The combination of the results of Sections \ref{sec:deformation}, \ref{sec:splitting},   and  \ref{sec:fundamental-groups} gives the proof of Theorem \ref{thm:main}.


\paragraph{Acknowledgments:}

The second author is partially supported by JSPS KAKENHI Grant Number JP18K03263 and JP23K03042. The third author is partially supported by JSPS KAKENHI Grant Number 21K03182. The fourth author is partially supported by SCE research grant 2023. The fifth author is partially supported by JSPS KAKENHI Grant Number 20K03561.   

\section{Construction of curves}\label{sec:construction}
In this section, we briefly describe how to construct the curves in the main theorem. The method is based on previous works of the second and fifth author and collaborators using elliptic surfaces and Gr\"obner basis techniques. See \cite{tokunaga2014},\cite{takahashi_tokunaga2020explicit},  \cite{bannai_tokunaga2021} for details.


\subsection{Mumford representation and Gr\"obner bases}

We give a brief summary about representations of divisors on hyperelliptic curves and
our method in constructing plane curves based on those representations. In this paper, we only work
in the case of elliptic curves, but we here give explanations  in general settings.

We refer to \cite{CLO} for the general theory of Gr\"obner bases. Also, as for details on 
semi-reduced or fully-reduced divisors on hyperelliptic curves and their representations
via Gr\"obner bases, we refer to \cite{galbraith} and \cite{takahashi_tokunaga2021}. 

Let $K$ be a perfect field of ${\mathrm{ch}}(K)\neq 2$ and let  $\Kbar$ denote its algebraic
closure. Let $\mcC$ be a hyperelliptic curve defined over $K$ given by
\[
\mcC: y^2 = F(x), \quad F(x) = x^{2g+1}+ c_1x^{2g} + \ldots + c_{2g+1}, \quad c_i \in K.
\]
$\mcC$ has a unique point at infinity, which we denote by $O$. We denote the hyperelliptic
involution $(x, y) \mapsto (x, -y)$ by $\iota$. We denote, also, the coordinate ring of $\mcC$ by
$\Kbar[\mcC]$ and its quotient field by $\Kbar(\mcC)$. For $P \in \mcC$, $\mcO_P$ denotes the local ring
at $P$ and $\ord_P$ means a valuation at $P$. For $g \in \Kbar[x,y]$, $[g]$ means its class in $\Kbar[\mcC]$.

 \begin{Def}\label{def:semi-reduced}
{\rm Let $\fd = \sum_{P\in \mcC}m_PP$ be a divisor on a hyperelliptic curve $\mcC$ . 
  
  \begin{enumerate}
  
  \item[(i)] The divisor $\fd$ is said to be an {\it affine divisor} if ${\mathrm{Supp}}(\fd) \subset \mcC_{{\mathrm{aff}}}:= \mcC\setminus \{O\}$.
  
  \item[(ii)] An effective affine divisor $\fd$ is said to be  {\it semi-reduced} if it satisfies the following conditions:
  \begin{enumerate}
       \item[(a)] $m_P = 1$ if $m_P > 0$ and $P = \iota(P)$, and
      \item[(b)] $m_{\iota(P)} = 0$ if $m_P > 0$ and $P\neq \iota(P)$. 
       \end{enumerate}

  \item[(iii)] 
  A semi-reduced divisor $\sum_i m_i P_i$ is said to be {\it fully-reduced} if $\sum_im_i \le g$.
  
  \end{enumerate}
  }
  \end{Def}

  \begin{rmk}{\rm In \cite{galbraith}, a fully-reduced divisor is simply called reduced. In this paper, we use fully-reduced to avoid confusion with {\it reduced} in the usual sense. Note that we use the terminology $h$-reduced for fully-reduced in 
  \cite{takahashi_tokunaga2021}.
  }
  \end{rmk}

 The following lemma is fundamental for semi-reduced and fully-reduced divisors:
 
\begin{lemma}\label{lem:semi-reduced1}{
\begin{enumerate}

 \item[\rm(a)]  For any divisor $\fd = \sum_Pm_PP$ with $\Supp(\fd) \neq \emptyset$, there exists a semi-reduced
divisor $\sr(\fd)$ such that 
 {\rm (i)} $\fd - (\deg \fd)O \sim \sr(\fd) - (\deg\sr(\fd))O$ and
{\rm (ii)} $|\fd| \ge |\sr(\fd)| (= \deg \sr(\fd))$.
Here  we put $\deg \fd := \sum_P m_P$ and
 $|\fd| := \sum_P|m_P|$.

\item[{\rm (b)}] Let $\fd$ be any semi-reduced divisor on $\mcC$ with $\deg \fd > g$. Then there exists a
unique fully-reduced divisor $\red(\fd)$ such that $\fd - (\deg\fd) O\sim \red(\fd) - (\deg{\red(\fd)})O$.

\item[{\rm (c)}]  With the two statements as above,  we see that for any element $\fd \in \Div^0(\mcC)$, there exists a unique
fully-reduced divisor $\red(\fd)$ such that $\fd \sim \red(\fd) - (\deg\red(\fd))O$.

\end{enumerate}
}
\end{lemma}
As for proofs, see \cite{galbraith}.

\begin{rmk}{\rm By Lemma~\ref{lem:semi-reduced1}, any element in $\pic^0(\mcC)$ is represented by some divisor of the form
$\fd - \deg \fd O$, where $\fd$ is a semi-reduced divisor. In hyperelliptic cryptography, the addition on $\pic^0(\mcC)$ is
described in terms of semi-reduced divisors. See \cite{galbraith}.
}
\end{rmk}
\begin{lemma}\label{lem:mumford}{
For a semi-reduced divisor $\fd=\sum_Pe_PP$, there exists a unique pair $(u, v)$ of
polynomials in $\Kbar[x]$ such that 
\begin{itemize}
    \item $u= \prod_{P\in \Supp(\fd)}(x - x_P)^{e_P}$,
    \item $\deg v < \deg u$, $y_P = v(x_P)$ and
    \item $\ord_P(y - v) \ge e_P$ for $\forall P \in \Supp(\fd)$.
\end{itemize}
In particular, $v^2 - F$ is divisible by $u$.
}
\end{lemma}
See \cite{galbraith} for a proof.

\begin{Def}\label{def:mumford}{\rm Let $\fd$ be a semi-reduced divisor as in Lemma~\ref{lem:mumford}. The pair $(u, v)$ in Lemma~\ref{lem:mumford} is
called the Mumford representation of $\fd$.
}
\end{Def}

Given a semi-reduced divisor $\fd = \sum_Pe_P P$, 
we define  ideals $I(\fd) \subset \Kbar[\mcC]$ and $\widetilde{I(\fd)} \subset \Kbar[x, y]$ as follows:
\begin{eqnarray*}
I(\fd) &:=& \{[g] \in \Kbar[\mcC] \mid \ord_P([g]) \ge e_P, \mbox{for $\forall P \in \Supp(\fd)$}\}, \\
\widetilde{I(\fd)} & := & \{ g \in \Kbar[x, y] \mid [g] \in I(\fd)\}
\end{eqnarray*}

\begin{prop}\label{prop:mumford}{Let $\fd$ be a semi-reduced divisor on $\mcC$ as in Lemma~\ref{lem:mumford} and let $(u, v)$ be its 
Mumford representation. Then $\{u, y - v\}$ is the reduced Gr\"obner basis of $\widetilde{I(\fd)}$ with respect to the pure lexicographic
order $>$ with $y>x$.
}
\end{prop}
See \cite[Proposition 2.8]{takahashi_tokunaga2021}. 

\begin{rmk}{\rm If a semi-reduced divisor $\fd$ is defined over $K$, $u, v \in K[x]$.
}
\end{rmk}



\subsection{Remark about the addition on an elliptic curve via
Mumford representation}\label{sub-sec:addition}

In the case of $g=1$, $\mcC$ is an elliptic curve. We use
$E$ instead of $\mcC$. Let $\fd = P_1 + P_2, \  P_i=(x_i, y_i), i = 1, 2$ be a semi-reduced divisor on $E$ of degree $2$. The Mumford
representation $(u_{\fd}, v_{\fd})$ is given by
\[
u_{\fd}= (x - x_1)(x - x_2), \quad v_{\fd} = mx + n,
\]
where $y = mx + n$ is the line $L_{\fd}$ connecting $P_1$ and $P_2$ 
(if $P_1 = P_2$, $L_{\fd}$ is the tangent line of $E$ at $P_1$). These can be
computed in the following way. As 
$I(\fd) = I(P_1)I(P_2)$ in $\Kbar[\mcC]$, we see that
\[
\widetilde{I(\fd)} = \langle u_{\fd}, y - v_{\fd} \rangle =
\langle (x - x_1)(x - x_2), (x - x_1)(y - y_2), (x - x_2)(y - y_1), 
(y - y_1)(y - y_2), y^2 -F \rangle.
\]
We now use Proposition~\ref{prop:mumford} to  compute $u_{\fd}, v_{\fd}$. Once $u_{\fd}$
and $v_{\fd}$ are obtained, we compute $(v_{\fd}^2 - f)/u_{\fd}$ and get the $x$-coordinate of 
the third intersection point between $\mcC$ and $L_{\fd}: y - v_{\fd} = 0$. 
Let $(x_3, y_3), \ y_3 = mx_3 + n$ be another point in $L_{\fd}\cap E$ ($(x_3, y_3)$ may
coincide with $P_1$ or $P_2$). 
Hence, 
$P_1\dot + P_2 = (x_3, - y_3)$, where $P_1\dot+P_2$ denotes the sum of $P_1$ and $P_2$ under the group law of $E$.


We now consider the case where $K=\CC(t)$ to construct curves with prescribed combinatorics. Given two points in $E(\CC(t))$, $P_i=(x_i(t), y_i(t))$ ($i=1, 2$), we can obtain additional new points by using the group law of $E$, such as $P_3=P_1\dot+P_2=(x_3(t), t_3(t))$ and $P_4=P_1\dot+[-1]P_2=(x_4(t), y_4(t))$. Here, the equations $x-x_i(t)=0$ ($i=1, 2, 3, 4$) give rise to plane curves in $(t, x)$-space and, in turn, in $\PP^2$. In this way, we see that the two curves $x-x_i(t)$ ($i=1, 2$)  produce new plane curves $x-x_i(t)$ ($i=3, 4$) which are related to the original curves. A detailed explanation of this construction is provided in Section \ref{sub-sec:construction}.

\subsection{Construction of the curves}\label{sub-sec:construction}
In this subsection we describe the elliptic surfaces and the computations of the equations that provide the curves in the main theorem. 
Consider a quartic $\mcQ=C_1+C_2$, which is a union of two smooth conics intersecting transversally.  
Let $z_o\in C_1$ be a point distinct from the intersection points. We can associate a rational elliptic surface $\varphi: S_{\mcQ, z_o}\rightarrow \PP^1$ to $E_{\mcQ}$ so that $S_{\mcQ, z_o}$ fits into the following commutative diagram:
\[
\begin{tikzcd}
S^\prime_{\mcQ} \arrow[d, "f^\prime_{\mcQ}"'] & S_{\mcQ}  \arrow[l,"\mu"'] \arrow[d, "f_{\mcQ}"] &  S_{\mcQ, z_o} \arrow[l, "\nu_{z_o}"'] \arrow[d, "f_{\mcQ, z_o}"]\\
\PP^2 & \widehat{\PP^2} \arrow[l,"q"] & (\widehat{\PP^2})_{z_o} \arrow[l, "q_{z_o}"]
\end{tikzcd}
\]
Here, $f'_{\mcQ}$ is the double cover of $\PP^2$ branched along $\mcQ$, $\mu$ is the resolution of singularities, and $\nu_{z_o}$ is the resolution of the indeterminacy of the pencil of genus 1 curves $\Lambda_{z_o}$ on $S_{\mcQ}$, which is the pencil induced by the pencil of lines through $z_o$ in $\PP^2$. Note that $\Lambda_{z_o}$ induces the structure of the elliptic fibration $\varphi: S_{\mcQ, z_o}\rightarrow \PP^1$. We will regard the exceptional divisor of the second blow-up in $\nu_{z_o}$ as the zero-section $O$.
Let $\MW(S_{\mcQ, z_o})$ be the Mordell-Weil lattice of $S_{\mcQ, z_o}$, which is the set of sections of $S_{\mcQ, z_o}$ endowed with a pairing $\langle, \rangle$ called the height pairing.
If  the tangent line of $C_1$ at $z_o$ intersects $C_2$ transversally and is a simple tangent line of $\mcQ$,  $S_{\mcQ, z_o}$ will have 5 singular fibers of type $\I_2$ and $\MW(S_{\mcQ, z_o})\cong (A_1^\ast)^{\oplus 3}\oplus \ZZ/2\ZZ$, by the list in Oguiso-Shioda \cite{oguiso-shioda}. If the tangent line of $C_1$ at $z_o$ is tangent to $C_2$ and is a bitangent of $\mcQ$, $S_{\mcQ, z_o}$ will have 1 singular fiber of type $\I_3$ and 4 singular fibers of type $\I_2$, and $\MW(S_{\mcQ, z_o})\cong \frac{1}{6}\left(\begin{array}{cc} 2 & 1 \\ 1 & 2 \end{array}\right)\oplus \ZZ/2\ZZ$, also by the list in \cite{oguiso-shioda}.

Let $[T, X, Z]$ be homogeneous coordinates of $\PP^2$ and let $t=\dfrac{T}{Z}$, $x=\dfrac{X}{Z}$ be affine coordinates. 
We can choose coordinates so that $z_o=[0,1,0]$, the tangent line at $z_o$ is $Z=0$, and the defining equation of the affine part of $C_1$ is $x-t^2=0$. Then, because $C_1$ and $C_2$ intersect transversally,  the defining equation of the affine part of $\mcQ$ can be assumed to be of the form 
\[
\mcQ: F(x, t)=(x-t^2)(x^2+a_1(t)x+a_2(t)),
\]
where $a_i(t)\in \CC[t]$ and $\deg_t a_i(t)\leq i$ ($i=1, 2$). We can consider an elliptic curve $E_{\mcQ, z_o}$ over $\CC(t)$ given by the Weierstrass equation $y^2=F(t, x)$.
Let $E_{\mcQ, z_o}(\CC(t))$ be the set of $\CC(t)$ rational points of $E_{\mcQ, z_o}$. It is known that there is a bijection between $\MW(S_{\mcQ, z_o})$ and $E_{\mcQ, z_o}(\CC(t))$. For $s\in \MW(S_{\mcQ, z_o})$ we denote the rational point corresponding to $s$ by $P_s$, and for $P\in E_{\mcQ, z_o}(\CC(t))$, we denote the section corresponding to $P$ by $s_P$. Under this correspondence, we have $s_{P_1\dot+P_2}=s_{P_1}\dot+s_{P_2}$. Furthermore, we denote the image $f^{\prime}_{\mcQ}\circ\mu\circ\nu_{z_o}(s)\subset \PP^2$ of a section $s$, by $\mcC_s$. Also,  for a rational point $P$, $\mcC_{P}:=\mcC_{s_P}$. Note that $\mcC_s$ is a curve if $s$ is not the zero-section and $\mcC_s=\mcC_{[-1]s}$. 

Let $C_1\cap C_2=\{p_1, p_2, p_3,  p_4\}$ and let the coordinates of  $p_i$ be $(t_i, t_i^2)$ $(i=1, 2, 3, 4)$. Let $L_{ij}$ be the line through $p_i, p_j$. Then, the affine defining equation of $L_{ij}$ is given by
\[
L_{ij}: x-(t_i+t_j) t+t_it_j=0.
\]
As $L_{ij}$ intersects $\mcQ=C_1+C_2$ at $p_i=(t_i, t_i^2), p_j=(t_j, t_j^2)$ each with multiplicity $2$, we see that 
\[
F(t, (t_i+t_j)t-t_it_j)=c_{ij}(t-t_i)^2(t-t_j)^2.
\]
Additionally, $L_{ij}$ gives rise to $\CC(t)$-rational points of $E_{\mcQ, z_o}$ of  the form
\[
\dot\pm P_{ij}=(x_{ij}, \pm y_{ij})=( (t_i+t_j)t-t_it_j, \pm d_{ij}(t-t_i)(t-t_j))\quad (d_{ij}^2=c_{ij}),
\]
which in turn correspond to sections $s_{ij}:=s_{P_{ij}}\in \MW(S_{\mcQ, z_o})$. We will use these rational points and sections to construct the curves that we consider in the main theorem, (i.e., bitangent lines and weak contact conics of $\mcQ$). First, we consider the weak contact conics.

\begin{lemma}\label{lem:wcc}
    Let $C$ be a smooth weak contact conic of $\mcQ=C_1+C_2$ passing through $p_i$ and $p_j$.
    If $C$ is tangent to $C_1$ at $z_o$ then $C=\mcC_{s_{ik}\dot+s_{kj}}$ or $C=\mcC_{s_{ik}\dot-s_{kj}}$ for $\{i, j, k\}\subset\{1, 2, 3, 4\}$ and $s_{ik}, s_{kj}\in \MW(S_{\mcQ, z_o})$. Moreover, for each choice of $\{p_i, p_j\}\subset \{p_1, p_2, p_3, p_4\}$ there are at most two smooth weak contact conics passing through $p_i, p_j$ and tangent to $C_1$ at $z_o$.
\end{lemma}
\begin{proof}
We only give a rough sketch of the proof. Let $C$ be a (weak) contact conic of $\mcQ=C_1+C_2$ tangent  to $C_1$ at $z_o$. We consider the elliptic surface $S_{\mcQ, z_o}$ associated to $\mcQ=C_1+C_2$ and $z_o$, as above. Here, $\MW(S_{\mcQ, z_o})\cong (A_1^\ast)^{\oplus 3}\oplus \ZZ/2\ZZ$ or $\frac{1}{6}\left(
\begin{array}{cc} 2 & 1 \\ 1 & 2 \end{array}
\right)\oplus \ZZ/2\ZZ$, depending on whether the tangent line at $z_o$ is a simple tangent or a bitangent of $\mcQ$. In the former case, $s_{12}, s_{23}, s_{31}$ gives a basis of the $(A_1^\ast)^{\oplus 3}$ part, while in the latter case, $s_{12}, s_{31}$ gives a basis of the $\frac{1}{6}\left(
\begin{array}{cc} 2 & 1 \\ 1 & 2 \end{array}
\right)$ part. The preimage of the strict transform of $C$ in $S_{\mcQ, z_o}$ gives a pair of  sections $\dot\pm s_C$ of $S_{\mcQ, z_o}$. The data of the intersection points $p_1, \ldots, p_4$ that $C$ passes through gives the data of the intersection of $\dot\pm s_C$ and the singular fibers of $S_{\mcQ, z_o}$ and the data of the values of height pairings with $s_{ij}$. Then, as we know the lattice structure of $\MW(S_{\mcQ, z_o})$, and the values of the height pairings of $\dot\pm s_C$ with the basis elements, we can deduce that $\dot\pm s_C=\dot\pm(s_{ik}\dot+s_{kj})$ or $\dot\pm(s_{ik}\dot-s_{kj})$. Finally, note that $s_{ik}\dot+s_{kj}=s_{il}\dot+s_{lj}$, hence we see that there are at most two possibilities for $C$.
See \cite{tokunaga2014, bannai-tokunaga-yamamoto2020, masuya2023} for details on similar arguments.
\end{proof}

Concerning the converse of Lemma \ref{lem:wcc}, we have the following:

\begin{lemma} \label{lem:wcc2}
 If the tangent line at $z_o$ is a simple tangent of $\mcQ$,
 the curves $\mcC_{s_{ik}\dot+s_{kj}}$ and $\mcC_{s_{ik}\dot-s_{kj}}$ are smooth weak contact conics passing through $p_i, p_j$ and tangent to $C_1$ at $z_o$. If the tangent line at $z_o$ is a bitangent of $\mcQ$,
 one of $\mcC_{s_{ik}\dot+s_{kj}}$ and $\mcC_{s_{ik}\dot-s_{kj}}$ will be a smooth weak contact conic passing through $p_i, p_j$ tangent to $C_1$ at $z_o$ and the other will coincide with $L_{ij}$.
\end{lemma}
\begin{proof}
The proof is given by following through the proof of Lemma \ref{lem:wcc} backwards by starting with the data of the height pairing and intersection with the singular fibers, to obtain the geometric data of $\mcC_{s_{ik}\dot\pm s_{kj}}$. Again, see \cite{tokunaga2014, bannai-tokunaga-yamamoto2020, masuya2023} for details on similar arguments.
\end{proof}

\begin{rmk}
By combining Lemma \ref{lem:wcc}, \ref{lem:wcc2} we see that if the tangent line at $z_o$ is a simple tangent of $\mcQ$ then there are exactly two weak contact conics satisfying the conditions of Lemma \ref{lem:wcc}.
When the tangent line at $z_o$ is a bitangent, 
the union of the tangent line at $z_o$ and the line $L_{ij}$ can be considered as a singular weak contact conic. Hence, in both cases the number will be exactly two, if we allow singular weak contact conics.
\end{rmk}
Next, we consider the bitangent lines.
\begin{lemma}\label{lem:bitangent}
Let $z_o\in C_1$ be a point whose tangent line is a simple tangent of $\mcQ$. Then the 4 bitangent lines of $\mcQ$ are given by $\mcC_{s_{12}\dot\pm s_{23}\dot \pm s_{31}}$. If the tangent line at $z_o$ is a bitangent, the remaining 3 bitangents are given by $\mcC_{s_{12}\dot\pm s_{23}\dot\pm s_{31}}$ for $s_{12}\dot\pm s_{23}\dot\pm s_{31}\not=O$.
\end{lemma}
\begin{proof}
This lemma can be proved by a similar argument as Lemma \ref{lem:wcc}. 
A detailed proof can be found in \cite{masuya2023}.  
\end{proof}

Now, we apply the above arguments to an explicit example to  obtain the equations given in the introduction. Consider the curves $C_1, C_2, \mcQ=C_1+C_2$, whose affine parts are given by the equations
\begin{align*}
     C_1: x-t^2=0, && C_2: x^2-10tx+25x-36=0, && \mcQ: F:=(x-t^2)(x^2-10tx+25x-36)=0.
\end{align*}
We put  $p_1=[3, 9, 1]$, $p_2=[2,4,1]$, $p_3=[6, 36, 1]$, and $p_4=[-1, 1, 1]$.
In this case, the tangent line at $z_o=[0,1,0]$ is a simple tangent and we have $\MW(S_{\mcQ, z_o})\cong (A_1^\ast)^{\oplus 3}\oplus \ZZ/2\ZZ$.
The affine parts of the lines $L_{ij}=\overline{p_ip_j}$ are given by the equations
\begin{align*}
L_{12}: x-5t+6=0, &&  L_{23}: x-8t+12=0, &&
L_{31}: x-9t+18,
\end{align*}
which give rise to $\CC(t)$-rational points 
\begin{align*}
    P_{12}=(5t-6, -5(t-2)(t-3)), &&  P_{23}=(8t-12, -4(t-2)(t-6)), &&
    P_{31}=(9t-18, -3(t-3)(t-6)),
\end{align*} 
which correspond to the generators $s_{12}, s_{{23}}, s_{{31}}$ of the  $(A_1^\ast)^{\oplus 3}$ part of $\MW(S_{\mcQ, z_o})$. Also, the equation of the conic $C_1$ gives rise to the torsion point $T=(t^2,0)$.
We can use the Gr\"obner basis techniques described in Subsection \ref{sub-sec:addition} to compute 
the addition of the above points and obtain the following rational points $Q_0, Q^\prime_0, Q_1, \ldots, Q_4$:
\begin{align*}
    Q_0&=P_{12}\dot+P_{31}=(x_{Q_0},y_{Q_0})=\left({\frac {5\,{t}^{2}}{4}}-2\,t+3,{\frac {5\,{t}^{3}}{8}}-6\,{t}^{2}+{\frac {31\,t}{2}}-12\right)\\
    Q_0^\prime&=P_{12}\dot-P_{31}=(x_{Q_0}^\prime,y_{Q_0}^\prime)=(5t^2-32t+48, 10t^3-114t^2+392t-408)\\
Q_1&=P_{12}\dot+P_{23}\dot+P_{31}=(x_{Q_1},y_{Q_1})=\left({\frac {32\,t}{5}}-{\frac{256}{25}},{\frac {24\,{t}^{2}}{5}}-{\frac {
726\,t}{25}}+{\frac{5472}{125}}\right)
\\
Q_2&=P_{12}\dot-P_{23}\dot+P_{31}=(x_{Q_2},y_{Q_2})=(0,-6\,t)\\
Q_3&=P_{12}\dot+P_{23}\dot-P_{31}=(x_{Q_3},y_{Q_3})=(10\,t-25,6\,t-30)\\
Q_4&=P_{12}\dot-P_{23}\dot-P_{31}=(x_{Q_4},y_{Q_4})=\left({\frac {18\,t}{5}}-{\frac{81}{25}},{\frac {24\,{t}^{2}}{5}}-{\frac {
474\,t}{25}}+{\frac{2322}{125}}\right)
\end{align*}
In turn, these rational points give rise to the curves $C_3=\mcC_{Q_0}$, $C^\prime_3=\mcC_{Q^\prime_0}$,  $L_1=\mcC_{Q_1}$, $L_2=\mcC_{Q_3}$, $L_3=\mcC_{Q_3}$, and $L_4=\mcC_{Q_3}$ 
whose equations become
\begin{align*}
    C_3:&\, u_0:=x-x_{Q_0}=x-\left(\frac{5}{4}t^2-2t+3\right)=0\\
    C^\prime_3: &\, u^\prime_0:=x-x_{Q^\prime_0}=x-\left(5t^2-32t+48\right)=0\\
    L_1:&\, u_1:=x-x_{Q_1}=x-\left(\frac{32}{5}t-\frac{256}{25}\right)=0\\
    L_2:&\, u_2:=x-x_{Q_2}=x=0\\
    L_3:&\, u_3:=x-x_{Q_3}=x-(10t-25)=0\\
    L_4:&\, u_4:=x-x_{Q_4}= x-\left(\frac{18}{5}t-\frac{81}{25}\right)=0.
\end{align*}
The combinatorics of these curves (i.e., that $C_3, C_3^\prime$ are weak contact conic passing through $p_2, p_3$ and that $L_i$ $i=1, 2, 3, 4$ are bitangents of $\mcQ$), can be deduced from Lemma \ref{lem:wcc} and \ref{lem:bitangent}, but can also be checked directly. These curves give the arrangements
\begin{align*}
    \mcC_i=\mcQ+C_3+L_i \quad (i=1, 2, 3, 4)
\end{align*}
of the main theorem, with combinatorial type $\Comb(\mcC)$.

\section{The realization space}\label{sec:deformation}

In this section, we study the realization space of the conic-line arrangements having the combinatorics ${\rm Comb}(\mcC)$ (i.e., we study the the quasi-projective variety 
in $\PP {\rm H}^0(\PP^2,\mathcal{O}(7))$ consisting of closed points corresponding to curves having the combinatorics $\Comb(\mcC)$).
If two arrangements $\mcC$ and $\mcC^\prime$ in this realization space can be deformed to each other and lie in the same connected component, we denote this by
\[
\mcC \sim \mcC^\prime.
\]

The main objective of this section is to prove that the realization space has exactly two connected components. We will prove this in the following three steps:
\begin{enumerate}
\item We will prove that any arrangement $\mcC=C_1+C_2+C_3+L$ with combinatorics $\Comb(\mcC)$ can be deformed while preserving the combinatorics to an arrangement $\mcC^\prime$ for some specific choice of $C_{1}^0$ and $C_2^0$. 
\item We will choose a specific $C_1^0$, $C_2^0$ and study curves with combinatorics $\Comb(\mcC)$ for this specific choice of $C_1^0, C_2^0$, which will give all of the representatives of the connected components. 
\item We will see how the representatives in Step (2) are related, and determine the  number of connected components.
\end{enumerate}
First we will show the following Lemma and Corollary to show Step (1):

\begin{lemma}
Let $C_1^0$, $C_2^0$ be smooth conics that intersect transversally. Let $z^0\in C_1^0$ be a point such that $z^0$ is distinct from the intersection points of $C_1^0$ and $C_2^0$, and the tangent line at $z^0$ is not a bitangent. Then any arrangement
$C_1+C_2+C_3$ of smooth conics having the combinatorics of the conics in $\Comb(\mcC)$ can be deformed while preserving the combinatorics to $C_1^0+C_2^0+C_3^0$, where $C_3^0$ is a weak contact conic tangent to $C_1^0$ at $z^0$. 
\end{lemma}

\begin{proof}\label{lem:deformation}
To prove this lemma,  we observe that the construction of the curves given in Section \ref{sub-sec:construction} can be applied to any pair of smooth conics $C_1, C_2$ that intersect transversally. Let $z$ be the tangent  point of $C_1$ and $C_3$. We can choose coordinates so that $z=[0:1:0]$, the tangent line at $z$ is  $Z=0$, and the defining equation of the affine part of $C_1$ is given by $x-t^2=0$. 
Let $\{p_1,\ldots, p_4\}=C_1\cap C_2$, then let $(t,x)=(t_i, t_i^2)$ $(i=1,\ldots, 4)$ be the coordinates of $p_i$ $(i=1,\ldots, 4)$ respectively. The defining equations of the lines $L_{ij}$ ($\{i, j\}\subset\{1, 2, 3, 4\})$ are given by
$L_{ij}: x-(t_i+t_j) t+t_it_j=0$.
Also, because $C_2$ passes through $p_1, \ldots, p_4$, $C_2$ is a member of the pencil  generated by $L_{12}+L_{34}$ and $L_{13}+L_{24}$ (i.e., there is $[u:v]\in\PP^1$ such that $C_2$ is given by
\[u\big(x-(t_1+t_2)t+t_1t_2\big)\big(x-(t_3+t_4)t+t_3t_4\big)+v\big(x-(t_1+t_3)t+t_1t_3\big)\big(x-(t_2+t_4)t+t_2t_4\big)=0). \]
Because $C_2$ intersects with $C_1$ transversally, 
a defining equation of $C_2$ is of the form
\[
C_2: x^2+a_1(t)x+a_2(t)=0,
\]
where $a_i(t)\in \CC[t]$, $\deg_t(a_i(t))\leq i$ and the coefficients of $a_i(t)$ depend continuously on $t_1, t_2, t_3, t_4, u, v$.
Now, we have a Weierstrass equation of the form 
\[
y^2=F(t, x)=(x-t^2)(x^2+a_1(t)x+a_2(t)),
\]
As explained in Section~\ref{sec:construction}, the lines $L_{ij}$ give rise to $\CC(t)$-rational points
\[
\dot\pm P_{ij}=(x_{ij}, \pm y_{ij})=( (t_i+t_j)t-t_it_j, \pm d_{ij}(t-t_i)(t-t_j))
\]
of $E(\CC(t))$, to which we can apply the construction of Section \ref{sec:construction}. Suppose that $C_3$ passes through $p_i, p_j$ $(i\not=j)$. Then by Lemma \ref{lem:wcc}, $C_3$ will become either $\mcC_{s_{ik}\dot+ s_{kj}}$ or $\mcC_{s_{ik}\dot-s_{kj}}$. The $x$-coordinate of $P_{ik}\dot \pm P_{kj}$ is given by
\[
\lambda_{\pm}^2-(a_1(t)-t^2)-x_{ik}-x_{kj},
\]
where 
\[\lambda_{\pm}=\frac{\pm y_{kj}-y_{ik}}{x_{kj}-x_{ik}}=\frac{(d_{ik}\mp d_{kj})t-d_{ik}t_i\pm d_{kj}t_j}{t_i-t_j}.\]
Now we see that the defining equation of $C_3$   given by
\[
C_3: x-\left(\lambda_{\pm}^2-(a_1(t)-t^2)-x_{ik}-x_{kj}\right)=0
\]
depends continuously on $t_1, t_2, t_3, t_4, u, v$, and that we can deform $C_3$ as we deform $C_2$. By Lemma \ref{lem:wcc2}, the deformation of $C_3$ is a weak contact conic as long as the line $Z=0$ is not a bitangent of $C_1$ and $C_2$. 

On the other hand, given $C_1^0, C_2^0$ and $z^0\in C_1^0$, there exists a projective transformation $\phi_0$ such that 
$\phi_0(z^0)=[0:1:0]\in \PP^2$, $\phi_0(C_1^0)=C_1$ and the tangent line at $z^0$ is transformed to $Z=0$. By the assumption on $z^0$, we see that $\phi_0(C_2^0)$ is not tangent to $Z=0$. Because the condition for the deformation of $C_2$  to be tangent to $Z=0$  is a closed condition for $(t_1, t_2, t_3, t_4, [u,v])\in \CC^4\times \PP^1$,  we can deform $C_2$ to $\phi_0(C_2^0)$ while not being tangent to $Z=0$. The composition of this deformation and $\phi_0^{-1}$  gives the desired deformation from $C_1+C_2+C_3$ to $C_1^0+C_2^0+C_3^0$. 
\end{proof}

\begin{rmk}\label{rmk:choice_of_points}
Note that when we deform $C_2$ to $\phi_0(C_2^0)$, 
we can choose freely to which points $\{p_1^0, p_2^0, p_3^0, p_4^0\}=\phi_0(C_1)\cap \phi_0(C_2)$ the points  $\{p_1, p_2, p_3, p_4\}$ will be deformed. 
\end{rmk}

\begin{cor}\label{cor:deformation}
Let $C_1+C_2+C_3+L$ be an arrangement with the combinatorics $\Comb(\mcC)$. Let $C_1^0$ and $C_2^0$ be smooth conics  intersecting transversally. Assume that there exists a point $z^0\in C_1^0$ such that every weak contact conic of $C_1^0+C_2^0$ tangent to $C_1^0$ at $z^0$ is not tangent to a bitangent line of $C_1^0+C_2^0$.
Then $C_1+C_2+C_3+L$ can be deformed while preserving the combinatorics to an arrangement $C_1^0+C_2^0+C_3^0+L^0$ having the combinatorics $\Comb(\mcC)$. 
\end{cor}
\begin{proof}
The fact that we can deform the bitangent lines $L_i$ ($i=1, 2, 3, 4$) as we deform $C_2$ as in  Lemma \ref{lem:deformation} can be seen by a similar argument. (Alternatively, 
we can consider the dual curves $C_1^\ast$, $C_2^\ast$ of $C_1$, $C_2$ respectively and observe that $C_2^\ast$ depends continuously on $C_2$ so that the four intersection points of $C_1^\ast$ and $C_2^\ast$ which correspond to the bitangent lines depend continuously on $C_2$.) We apply the deformation of Lemma \ref{lem:deformation} to deform $C_1+C_2+C_3$ to $C_1^0+C_2^0+C_3^0$. Since the coefficients of the defining equations of the deformation of $C_3$ and the bitangents are given in terms of $t_1, t_2, t_3, t_4, u, v$, the condition for a deformation of $C_3$ to be tangent to a deformation of a bitangent $L_i$ is a closed condition in  $(t_1, t_2, t_3, t_4, [u,v])\in \CC^4\times \PP^1$, so we can deform $C_2$ to $\phi_0(C_2^0)$ as in Lemma \ref{lem:deformation} while preserving the combinatorics.
\end{proof}

Next, for step (2), we choose a specific $C_1, C_2$ and use it to analyze the entire representation space. Let 
\begin{align*}
C_1: t^2+x^2+tx-\frac{27}{4}=0 && C_2: t^2+x^2-tx-\frac{27}{4}=0 && \mcQ=C_1+C_2.
\end{align*}
The bitangent lines of $\mcQ$ are $L_1: t=3$, $L_2: t=-3$, $L_3: x=3$, and $L_4: x=-3$. Let, $C_1\cap C_2=\{p_1, p_2, p_3, p_4\}$ and $p_1=[0, \frac{3}{2}\sqrt{3},1]$, $p_2=[-\frac{3}{2}\sqrt{3}, 0,1]$, $p_3=[\frac{3}{2}\sqrt{3}, 0,1]$, and $p_4=[0, -\frac{3}{2}\sqrt{3},1]$. For each point $z_a\in C_1$, there exist two weak contact conics $C_{3, a}$ and $C^\prime_{3, a}$ passing through $p_2, p_3$ and are tangent to $C_1$ at $z_a$ by Lemma \ref{lem:wcc}. These curves are obtained from  the rational points $P_{12}\dot+P_{31}$ and  $P_{12}\dot-P_{31}$. Hence, $C_1+C_2$ has two families $\{C_{3, a}\}$ and $\{C^\prime_{3, a}\}$ of weak contact conics passing through $p_2$ and $p_3$. Under the parametrization $z_a=\left(-\frac{3(a^2+4a+1)}{2(a^2+a+1)}, -\frac{3(a^2-2a-2)}{2(a^2+a+1)} \right)$, the defining equations of the members of the families are given by
\begin{align*}
C_{3, a}: &\left( 4\,{a}^{2}+8\,a \right) {t}^{2}+ \left( 8\,a+4 \right) {x}^{2}
+ \left( -12\,{a}^{2}-12\,a-12 \right) x-27\,{a}^{2}-54\,a=0,
\\
C^\prime_{3, a}: &\left( 4\,{a}^{2}-4 \right) {t}^{2}+ \left( -4\,{a}^{2}-16\,a-4
 \right) tx+ \left( -4\,{a}^{2}+4 \right) {x}^{2}+ \left( -24\,{a}^{2}
-24\,a-24 \right) x-27\,{a}^{2}+27=0.
\end{align*}  
Note that when $a=-2, -1, 0, 1$, the tangent line at $z_a$ is a bitangent, and the curves $C_{3,a}$ ($a=-2, 0$) and  $C^\prime_{3, a}$ ($a=-1,1$) degenerate to a union of $L_{23}$ given by $x=0$ and some bitangent line. Also, when $a=-2\pm\sqrt{3}, 1\pm\sqrt{3}$, $z_a$ coincides with one of $p_1, p_2, p_3, p_4$ and the combinatorics will become degenerated. 

It can be easily checked that $C_1+C_2$ satisfies the conditions for $C_1^0, C_2^0$ in Corollary \ref{cor:deformation}. Hence, by Corollary \ref{cor:deformation} and Remark \ref{rmk:choice_of_points}, any arrangement  with the combinatorics ${\rm Comb}(\mcC)$ can be deformed to a curve of the form $C_1+C_2+C_{3, a}+L_i$ or $C_1+C_2+C^\prime_{3, a}+L_i$ for some $a\in \CC$ and $i=1, 2, 3, 4$ for this specific choice of $C_1$ and $C_2$.

Finally, for step (3), we consider the relation of the above curves with combinatorics $\Comb(\mcC)$. Because 
$\{C_{3, a}\}$ and $\{C^\prime_{3, a}\}$ are connected families, 
we have
\begin{align*}
C_1+C_2+C_{3,a}+L_i \sim C_1+C_2+C_{3,a^\prime}+L_i,\\
C_1+C_2+C^\prime_{3,a}+L_i \sim C_1+C_2+C^\prime_{3,a^\prime}+L_i
\end{align*}
for any $i=1, 2, 3, 4$ and $a, a^\prime\in \CC$, such that the arrangement has combinatorics $\Comb(\mcC)$. The arguments so far show that we have $8$ representatives of connected components and that the number of connected components is less than or equal to $8$. 

Next, to study the relation of the above $8$ possibilities, we consider a deformation of $C_2$. Let $b$ be a parameter and $C_{2, b}$ be a conic defined by
\begin{align*}
C_{2, b}: &-27\,{b}^{4}-54\,{b}^{3}-81\,{b}^{2}-54\,b-27+ \left( 8\,{b}^{4}+16\,{
b}^{3}-24\,{b}^{2}-32\,b-4 \right) tx\\
&+ \left( 4\,{b}^{4}+8\,{b}^{3}+12
\,{b}^{2}+8\,b+4 \right) {t}^{2}+ \left( 4\,{b}^{4}+8\,{b}^{3}+12\,{b}
^{2}+8\,b+4 \right) {x}^{2}=0.
\end{align*}
Here, $C_{2, b}$ passes through $p_1, p_2, p_3, p_4$ and furthermore, $C_{2,b}=C_2$ for $b=-2, -1, 0, 1$.
Also, $C_{2, b}=C_1$ if $(b^2+4b+1)(b^2-2b-2)=0$ and $C_{2, b}$ is singular if $b^2+b+1=0$. When $(b^2+4b+1)(b^2-2b-2)(b^2+b+1)\not=0$, 
the curve $\mcQ_b:=C_1+C_{2, b}$ has the following bitangents $L_{1, b}, \ldots, L_{4, b}$, and a weak contact conic $D_{3, b}$ passing through $p_2, p_3$:
\begin{align*}
L_{1, b}: &  \left( {b}^{2}+2\,b \right) t+ \left( {b}^{2}-1 \right) x+3\,{b}^{2}+
3\,b+3=0 \\
L_{2, b}: & \left( {b}^{2}+2\,b \right) t+ \left( {b}^{2}-1 \right) x-3\,{b}^{2}-
3\,b-3=0 \\
L_{3, b}: &  \left( {b}^{2}-1 \right) t+ \left( {b}^{2}+2\,b \right) x+3\,{b}^{2}+
3\,b+3=0 \\
L_{4, b}: &  \left( {b}^{2}-1 \right) t+ \left( {b}^{2}+2\,b \right) x-3\,{b}^{2}-
3\,b-3=0 \\
D_{3, b}: &  \left( -20\,{b}^{2}+24\,b+32 \right) {t}^{2}+ \left( -52\,{b}^{2}-104
\,b \right)tx \\
&+ \left( -32\,{b}^{2}-24\,b+20 \right) {x}^{2}+ \left( -
84\,{b}^{2}-336\,b-84 \right) x+135\,{b}^{2}-162\,b-216=0
\end{align*}
Here, $D_{3, b}$ is singular if $b=2, 1\pm\sqrt{3}, -\frac{4}{5}$.
For $b=-2, -1, 0, 1$, we have the following Table \ref{table:curves}:

\begin{table}[h]
\centering
\begin{tabular}{|c||c|c|c|c|}
\hline
$b$ & $-2$ & $-1$ & $0$ & $1$   \\
\hline
$L_{1,b}$ & $L_2$   & $L_3$ & $L_1$& $L_4$\\
\hline
$L_{2, b}$ & $L_1$ & $L_4$ & $L_2$& $L_3$\\
\hline
$L_{3, b}$ & $L_4$ & $L_1$ & $L_3$& $L_2$\\
\hline
$L_{4, b}$ & $L_3$ & $L_2$  & $L_4$& $L_4$\\
\hline
$D_{3, b}$ & $C_{3,2}$ & $C^\prime_{3,2}$ & $C_{3, 2}$ & $C^\prime_{3, 2}$.\\
\hline
\end{tabular}
\caption{The curves given by $L_{i, b}$ and $D_{3, b}$.}
\label{table:curves}
\end{table}
Note that the equations of the curves may vary by a constant, this means that $L_1=L_{1, 0}=L_{2, -2}$ as curves but the equations of $L_{1, 0}$, $L_{2, -2}$ given by $b=0, -2$ respectively differ by a constant. 
By considering $C_1+C_{2, b}+D_{3, b}+L_{1, b}$ and $C_1+C_{2, b}+D_{3, b}+L_{3, b}$ for $b=0,-2, -1, 1$  we see that 
\begin{align*}
C_1+C_2+C_{3, 2}+L_1 \sim C_1+C_2+C_{3, 2}+L_2\sim C_1+C_2+C^\prime_{3, 2}+L_3 \sim C_1+C_2+C^\prime_{3, 2}+L_4,\\
C_1+C_2+C_{3, 2}+L_3 \sim C_1+C_2+C^\prime_{3, 2}+L_1\sim C_1+C_2+C_{3, 2}+L_4 \sim C_1+C_2+C^\prime_{3, 2}+L_2
\end{align*}
because we can deform while avoiding the finite number of exceptional values of $b$ where the combinatorics become  degenerated. Therefore, the deformation space has at most two connected components. On the other hand, we will see in the next section that there exists a Zariski pair of arrangements with combinatorics ${\rm Comb}(\mcC)$ whose curves cannot be in the same component. Hence,   the number of connected components of the deformation space must be exactly two.

\section{Splitting types}\label{sec:splitting}

In this section, we give two methods to compute the splitting types of the triples $(C_3,L_i; \mcQ)$ ($i=1,\dots,4$) with respect to the double cover $f'_\mcQ:S'_\mcQ\to\PP^2$. The first is conceptual in nature, where we only need the data of the combinatorics of sections $s\in\MW(S_{\mcQ, z_o})$ and corresponding curves $\mcC_s$, to calculate the splitting type. The second is more computational, where we consider explicit equations coming from the coordinates of corresponding rational points $P_s\in E_{\mcQ}(\CC(t))$ and use Gr\"obner basis techniques. Both methods have advantages and disadvantages, and we believe it worthwhile to present both methods.

The definition of the splitting type is as follows:

\begin{Def}[{\cite{bannai2016}}]
Let $\phi:X\to\PP^2$ be a double cover branched at a plane curve $\mcB$, and
let $D_1, D_2\subset\PP^2$ be two irreducible curves such that $\phi^\ast D_i$ are reducible and $\phi^\ast D_i=D_i^++D_i^-$.
For integers $m_1\leq m_2$, we say that the triple $(D_1, D_2;\mcB)$ has a {\it splitting type} $(m_1, m_2)$ if for a suitable choice of labels $D_1^+\cdot D_2^+=m_1$ and $D_1^+\cdot D_2^-=m_2$.
\end{Def}

The following proposition enables us to distinguish the embedded topology of plane curves by the splitting type.

\begin{prop}[{\cite[Proposition~2.5]{bannai2016}}]
Let $\phi_i:X_i\to\PP^2$ $(i=1,2)$ be two double covers branched along plane curves $\mcB_i$, respectively.
For each $i=1,2$, let $D_{i1}$ and $D_{i2}$ be two irreducible plane curves such that $\phi_i^\ast D_{ij}$ are reducible and $\phi_i^\ast D_{ij}=D_{ij}^++D_{ij}^-$.
Suppose that $D_{i1}\cap D_{i2}\cap \mcB_i=\emptyset$, $D_{i1}$ and $D_{i2}$ intersect transversally, and that $(D_{11},D_{12};\mcB_1)$ and $(D_{21}, D_{22};\mcB_2)$ have distinct splitting types.
Then there is no homeomorphism $h:\PP^2\to\PP^2$ such that $h(\mcB_1)=\mcB_2$ and $\{h(D_{11}), h(D_{12})\}=\{D_{21}, D_{22}\}$.
\end{prop}

\subsection{Conceptual calculation}
In this subsection, we describe how to compute the splitting types of the curves in the main theorem by using the height pairing of elliptic surfaces. By the construction of the curves given in Subsection \ref{sub-sec:construction}, for $s_{Q_i}, i=0, 1, \ldots, 4$ we have
\begin{gather*}
    (f^\prime_{\mcQ})^{-1}(C_{{Q_i}})=C_{Q_i}^++C_{Q_i}^-,\quad C^\pm_{Q_i}=\mu\circ\nu_{z_o}(s_{[\pm1]Q_i})
\end{gather*}
and because the curves $C_3, L_1, L_2, L_3, L_4$ do not intersect at singular points of $\mcQ$ or at $z_o$, it is enough to compute the intersection numbers $[\pm1]s_{Q_0}\cdot[\pm1] s_{Q_j}$ in $S_{\mcQ, z_o}$  to obtain the necessary splitting types.
Similar calculations have been done in \cite{bannai2016}, \cite{bannai_tokunaga2021} and we refer the reader to these papers for details.

First, we recall the following explicit formula of the height pairing for 
{{$P_1, P_2 \in E_{\mcQ, z_o}(\CC(t))$}}:
{\[
\langle P_1, P_2\rangle=\chi(S)+ s_{P_1}\cdot O+s_{P_2}\cdot O-s_{P_1}\cdot s_{P_2}-\sum_{v\in{\rm Red}(\varphi)} {\rm Contr}_v(P_1,P_2)
\]}
where 
{{\[
\mbox{Contr}_v(P_1, P_2) = {}^t\bm{c}(v, s_{P_1})(-A_v)^{-1}
\bm{c}(v,s_{P_2}),
\]}}
and
\[          
 \bm{c}(v, s) :=\left(
 \begin{array}{c}
 s\cdot\Theta_{v,1}\\
 \vdots \\
   s\cdot \Theta_{v,m_v-1}
 \end{array} \right)
\]
for $s \in \MW(S_{\mcQ,z_o})$.
Explicit values for {{${\rm Contr}_v(P_1, P_2)$}} can be found in \cite{shioda90}. In our case, $\chi(S_{\mcQ, z_o})=1$ and the values of  {{$s_{Q_i}\cdot O$}},  and  {{$\Contr_v(Q_i, Q_j)$}} for the sections that we use can be calculated from the geometric data that can be read off  from the construction of the elliptic surface $S_{\mcQ, z_o}$. Hence, it is enough to know the value of the height pairing {{$\langle [\pm 1]Q_0, [\pm 1]Q_i\rangle$}} to calculate $[\pm1]s_{Q_0}\cdot [\pm1]s_{Q_i}$.

Let $F_1, \ldots, F_4, F_{\infty}$ be the five singular fibers of type $\I_2$ of $S_{\mcQ, z_o}$, where $F_i$ corresponds to the line $\overline{z_op_i}$ ($i=1,2,3,4$) and $F_\infty$ corresponds to the tangent line of $C_1$ at $z_o$. We will label the components as $F_i=\Theta_{i,0}+\Theta_{i,1}$, where $\Theta_{i,0}\cdot O=1$ ($i=1,2,3,4, \infty$),  $\Theta_{i,0}$ is the strict transform of  $\overline{z_op_i}$ $(i=1, 2, 3, 4)$ and $\Theta_{\infty,1}$ is the strict transform of the tangent line at $z_o$. 

From the construction, we have 
\begin{gather*}
    s_{Q_i}\cdot O=0 \quad (i=0,1,2,3,4)\\
    s_{Q_0}\cdot \Theta_{\infty,0}=   s_{Q_0}\cdot\Theta_{1,0}=s_{Q_0}\cdot\Theta_{2,1}=s_{Q_0}\cdot\Theta_{3,1}=s_{Q_0}\cdot\Theta_{4,0}=1\\
    s_{Q_i}\cdot \Theta_{j,0}=s_{Q_i}\cdot\Theta_{\infty,1}=1, (i, j=1,2,3,4).
\end{gather*}
These values give
\[
{{\Contr_{v_k}(Q_o, Q_i)=0}},\quad (k=0,1,\ldots, 4, i=1,\ldots, 4),
\]
where $v_k$ corresponds to $F_k$, $k=1,\ldots, 4, \infty$.
Also, because $\MW(S_{\mcQ, z_o})\cong (A_2^\ast)^{\oplus 3}\oplus\ZZ/2\ZZ$, $s_{P_{12}}$, $s_{P_{23}}$, and  $s_{P_{31}}$ are generators of the $(A_1^\ast)^{\oplus 3}$,  
we have
\begin{align*}
    \langle Q_0, Q_1\rangle=\langle Q_0, Q_2\rangle=1\\
    \langle Q_0, Q_3\rangle=\langle Q_0, Q_4\rangle=0.
\end{align*}
Hence, by substituting these values into the explicit formula of the height pairing, we obtain
\begin{align*}
    s_{Q_0}\cdot s_{Q_1}=s_{Q_0}\cdot s_{Q_2}=0 \\
    s_{Q_0}\cdot s_{Q_3}=s_{Q_0}\cdot s_{Q_4}=1,
\end{align*}
which shows:
\begin{itemize}
\item  $(C_3, L_1; \mcQ)$ and $(C_3, L_2; \mcQ)$ have splitting type $(0,2)$, 
\item $(C_3, L_3; \mcQ)$ and $(C_3, L_4; \mcQ)$ have splitting type $(1,1)$. 
\end{itemize}

\subsection{Computational verification}

In this subsection, we use direct computation to verify the splitting types computed in the previous subsection.

Let $f_{\mcQ}':S_{\mcQ}'\to \PP^2$ be the double cover branched at $\mcQ:=C_1+C_2$ as Subsection~\ref{sub-sec:construction}. 
Over the affine open set $\{Z\ne0\}\subset\PP^2$, the double cover $S_{\mcQ}'$ is locally defined by $y^2=F$ in $\CC^3$, where $(t,x,y)$ is a system of coordinates of $\CC^3$. 
Because $C_3$ and $L_i$ are rational curves that are tangent to the branch locus $\mcQ$, the pull-backs $(f'_\mcQ)^\ast C_3$ and $(f'_\mcQ)^\ast L_i$ consist of the components  
\[ (f'_\mcQ)^\ast C_3=C_3^++C_3^-, \qquad (f'_\mcQ)^\ast L_i=L_i^++L_i^-. \] 

We compute the splitting types of $(C_3,L_i; \mcQ)$. 
Because $C_3$ and $L_i$ intersect transversally in the affine open $\{Z\ne0\}$, 
it is enough to compute the number of intersection points of $C_3^+$ and $L_i^\pm$ over $\{Z\ne 0\}$. 
By Proposition~\ref{prop:mumford} and the coordinates of $Q_i=(x_{Q_i},y_{Q_i})$ in Subsection~\ref{sub-sec:construction}, the defining ideals $\widetilde{I(C_3^\pm)}$ and $\widetilde{I(L_i^\pm)}$ of $C_3^\pm$ and $L_i^\pm$ as subvarieties in $\CC^3$ are as follows:
\begin{align*}
    \widetilde{I(C_3^\pm)}&=\langle x-x_{Q_0},y\mp y_{Q_0}\rangle =\left\langle x-\left({\frac {5\,{t}^{2}}{4}}-2\,t+3\right), y\mp \left({\frac {5\,{t}^{3}}{8}}-6\,{t}^{2}+{
\frac {31\,t}{2}}-12\right)\right\rangle\\
\widetilde{I(L_1^\pm)}&=\langle x-x_{Q_1},y\mp y_{Q_1}\rangle=\left\langle x-\left({\frac {32\,t}{5}}-{\frac{256}{25}}\right), y\mp \left({\frac {24\,{t}^{2}}{5}}-{\frac {
726\,t}{25}}+{\frac{5472}{125}}\right)\right\rangle
\\
\widetilde{I(L_2^\pm)}&=\langle x-x_{Q_2},y\mp y_{Q_2}\rangle=\langle x,y\pm 6\,t\rangle\\
\widetilde{I(L_3^\pm)}&=\langle x-x_{Q_3},y\mp y_{Q_3}\rangle=\langle x-(10\,t-25), y \mp (6\,t-30)\rangle\\
\widetilde{I(L_4^\pm)}&=\langle x-x_{Q_4},y \mp y_{Q_4}\rangle=\left\langle x-\left({\frac {18\,t}{5}}-{\frac{81}{25}}\right),y\mp \left({\frac {24\,{t}^{2}}{5}}-{\frac {
474\,t}{25}}+{\frac{2322}{125}}\right)\right\rangle
\end{align*}

The set of intersection points of $C_3^+$ and $L_i^\pm$ is defined by the ideal 
\[ I_i^\pm=\langle x-x_{Q_0}, y+y_{Q_0}, x-x_{Q_i}, y\mp y_{Q_i} \rangle\subset\CC[t,x,y] \qquad (i=1,\dots,4). \]
Because $I_i^\pm$ are zero-dimensional ideals, the number of intersection points of $C_3^+$ and $L_i^\pm$ is equal to the dimension of the $\CC$-vector space $\CC[t,x,y]/I_i^\pm$ for each $i=1,\dots,4$. 
A Gr\"obner basis $G_i^\pm$ of each $I_i^\pm$ with respect to the lex order with $x>y>t$ is as follows;
\begin{align*}
	G_1^+&=\{1\}, 
	\\
	G_1^-&=\{125 t^{2}-840 t +1324, \ 625 y +2010 t -4416, \ 25 x -160 t +256\};
		\\[0.5em]
	G_2^+&=\{1\};
	\\
	G_2^-&=\{5 t^{2}-8 t +12, \ y -6 t, \ x\}, 
		\\[0.5em]
	G_3^+&=\{ t -4, \ y +6, \ x -15\},
	\\
	G_3^-&=\{ 5 t -28, \ 5 y +18, \ x -31 \};
	\\[0.5em]
	G_4^+&=\{ 25 t -52, \ 3125 y +294, \ 125 x -531 \},
	\\
	G_4^-&=\{ 5 t -12, 25 y +18, 5 x -27 \}.
\end{align*}
Therefore the splitting types are as follows;
\begin{itemize}
\item  $(C_3, L_1; \mcQ)$ and $(C_3, L_2; \mcQ)$ have splitting type $(0,2)$, 
\item $(C_3, L_3; \mcQ)$ and $(C_3, L_4; \mcQ)$ have splitting type $(1,1)$. 
\end{itemize}

\section{Fundamental Groups}\label{sec:fundamental-groups}

In the above sections, we have studied the embedded topology of the four conic-line arrangements $\mcC_i$ (i=1,2,3,4) of degree $7$ consisting of $C_1, C_2, C_3$ and $L_i$ each.

As stated in the introduction, the fundamental group $\pi_1(\mathbb{P}^2 \setminus{\mcC})$ of the complement of a plane curve $\mcC \subset \PP^2$  has been used to study the embedded topology of plane curves. We can understand whether  two plane curves $({\mcC}_1, {\mcC}_2)$ with the same combinatorics form a Zariski pair with $\pi_1(\mathbb{P}^2 \setminus {\mcC}_1) \ncong \pi_1(\mathbb{P}^2 \setminus {\mcC}_2)$. Contrary to that, in a case where $\pi_1(\mathbb{P}^2 \setminus {\mcC}_1) \cong \pi_1(\mathbb{P}^2 \setminus {\mcC}_2)$, the curves ${\mcC}_1$ and ${\mcC}_2$ are called $\pi_1$-equivalent Zariski pairs.

In this section, we conclude the proof of Theorem \ref{thm:main}. We study two curves, denoted as  ${\mcC}_i$ ($i\in \{1,3\}$), each of which is of degree $7$ with three smooth conics $C_1, C_2, C_3$,  and line $L_i$ ($i\in \{1,3\}$); in each curve, line $L_i$ is tangent to conics $C_1, C_2$ and intersects conic $C_3$. We note that the curves ${\mcC}_1$ and ${\mcC}_2$ lie in the same connected component
and so do ${\mcC}_3$ and ${\mcC}_4$. Therefore we do not calculate the fundamental groups associated to  ${\mcC}_2$ and ${\mcC}_4$.
In Subsections \ref{sec:C1} and \ref{sec:C3} we determine the fundamental groups  $\pi_1(\mathbb{P}^2 \setminus {\mcC}_i)$ ($i=1, 3$).
Because we get that both fundamental groups are free abelian on $3$ generators, we conclude that ${\mcC}_1$ and ${\mcC}_3$ are $\pi_1$-equivalent, which finishes the proof of Theorem \ref{thm:main}.

Figures \ref{curveC1} and \ref{curveC3} depict  projective transformations of curves ${\mcC}_1$ and ${\mcC}_3$. Each curve has four types of singularities: nodes and tangency points (between the line and a conic, or between two conics), branch points of the conics, and intersection points of the three conics together.

To compute the fundamental group $ \pcpt{\mathcal{C}_i} $ for a curve $ \mathcal{C}_i $ we use the Zariski-Van Kampen algorithm as described in \cite{vanKampen33} which, for the sake of completeness, we describe here.
The work of Cogolludo about monodromy and fundamental groups \cite{Cog1} is a great basis for understanding our explanations, and later on, understanding the computations as well.

We begin by computing the affine fundamental group $ \pi_1(\mbb{C}^2 \setminus \mathcal{C}_i, *) $; that is, we pick a generic line $ L\subseteq \cpt $ and choose coordinates such that $ L $ is the line at infinity.
We then consider the projection $ \pr: \mbb{C}^2 \to\mbb{C}^1 $ given by $ (x,y)\mapsto x $.
The genericity conditions ensure that no tangent line to $ \mathcal{C}_i $ at a singularity can be parallel to the $ y $-axis.
Let $ q_1,\dots, q_N\in \mathbb{C}^1 $ be the branch locus of $ \pr|_{\mathcal{C}_i} $, that is, the images of the singularities of $ \mathcal{C}_i $ and the images of points of $ \mathcal{C}_i $ where the tangent to $ \mathcal{C}_i $ is parallel to the $ y $-axis (the latter points are called \emph{branch points}).
Pick a base point $ y_0\in \mathbb{C}^1\setminus \{q_1,\dots,q_N\} $ and a base point $ x_0\in \pr^{-1}(y_0)\setminus \mathcal{C}_i $ in its fiber.
One can be convinced that any loop in $ \pi_1(\CC^2 \setminus \mathcal{C}_i, x_0 ) $ is equivalent to a loop whose image under $ \pr $ avoids points $ q_1,\dots, q_N $.
Covering $ \CC^1-\{q_1,\dots,q_N\} $ by simply-connected open neighborhoods of $ y_0 $ and using the Van Kampen theorem, we see that any loop in $ \pi_1(\CC^2\setminus \mathcal{C}_i, x_0) $ is in fact equivalent to a loop that is entirely contained in the fiber $ \pr^{-1}(y_0) \setminus \mathcal{C}_i $.
So, the fundamental group $ \pi_1(\CC^2\setminus \mathcal{C}_i, x_0) $ is a quotient of $ \pi_1(\pr^{-1}(y_0)\setminus \mathcal{C}_i, x_0) $,  which is isomorphic to the free group on $ \deg \mathcal{C}_i $ generators.
To find the relations that define $ \pcpt{\mathcal{C}_i} $, we consider the monodromy action of $ \pi_1(\CC^1 \setminus \{ q_1, \dots, q_N \}, y_0) $ on $ \pi_1(\pr^{-1}(y_0)\setminus \mathcal{C}_i, x_0) $.
For every element $ [\gamma]\in \pi_1(\CC^1 \setminus \{q_1,\dots, q_N\}, y_0) $ and $ [\Gamma]\in \pi_1(\pr^{-1}(y_0) \setminus \mathcal{C}_i, x_0) $, if we denote by $ [\gamma]\cdot[\Gamma]\in\pi_1(\pr^{-1}(y_0)\setminus \mathcal{C}_i, x_0)$, the result of the monodromy action of $ [\gamma] $ on $ [\Gamma] $, then this action induces homotopy (in $ \CC^2 \setminus \mathcal{C}_i $) between $ [\gamma]\cdot[\Gamma] $ and $ [\Gamma] $ so they are equal in $ \pi_1(\CC^2\setminus \mathcal{C}_i, x_0) $.
In fact, those are all the relations in $ \pi_1(\CC^2 \setminus \mathcal{C}_i, x_0) $.
To get a representation of $ \pcpt{\mathcal{C}_i} $, one can use the Van Kampen theorem again, which gives one additional relation, called a \emph{projective relation}. This projective relation corresponds to the fact that a loop around all the points in $ \mathcal{C}_i\cap \text{pr}^{-1}(y_0) $ is null-homotopic in $ \mathbb{CP}^2 \setminus \mathcal{C}_i $.
The projective relation will always be the product of all the generators in the order they appear in the fiber $\pr^{-1}(y_0)$.

Obviously, it is enough to consider the relations arising from a generating set of the group $ \pi_1(\mathbb{C}^1 \setminus \{q_1,\dots,q_N\}, y_0) $.
In all our cases we are able to pick coordinates such that $ q_1,\dots,q_N $ are all real. We then pick $y_0$ to be real and larger than $ \max\{ q_1, \dots, q_N \} $.
We choose generating set $ [\gamma_1],\dots,[\gamma_N] $ of $ \pi_1(\mathbb{C}^1 \setminus \{q_1,\dots,q_N\}, y_0) $, such that $ \gamma_i $ is a loop that goes in the upper half plane to $ q_i $, then performs a counter clockwise twist around $ q_i $, and finally returns to $ y_0 $ in the upper half plane.
The calculation of the monodromy action is then separated into a local calculation around $ q_i $ and a conjugation (referred to as a \emph{diffeomorphism}) corresponding to replacing the basepoint from some point that lies close to $ q_i $, to the point $ y_0 $.
The data of the local relations and diffeomorphisms corresponding to all the singularities $ q_i $ is represented in a \emph{monodromy table},  see for example Table \ref{table:monodromy_table_C1_example}.

Given a branch point, a node, or a cusp, we define a skeleton $\langle i , i+1 \rangle $ to be a vertical line segment connecting points $i$ and $i+1$ that are positioned on the two components that meet at the singularity. To understand it more clearly, we look at the left side of Figure \ref{fig:skeleton_diffeomorphism_example}; we can see a fiber with seven points (that are the intersections of this fiber with a curve of degree $7$), and the skeleton $ \langle 6,7 \rangle $ that connects the points on the fiber, numerated as $6$ and $7$.

Given an intersection point of three components (e.g., intersection point of three conics) we have a skeleton of the form $ \langle i , i+1 , i+2 \rangle $.

\begin{table}[ht]
	\begin{center}		
		\begin{tabular}{| c | c | c | c |}
			\hline \hline
			Vertex number& Vertex description& Skeleton& Diffeomorphism\\ \hline \hline
			1& $C_1$ branch& $ \langle 1 - 2 \rangle $& $ \Delta_{I_{4}I_{6}}^{1/2}\langle 1 \rangle $ \\ \hline
			2& $C_2$ branch& $ \langle 2 - 3 \rangle $& $ \Delta_{I_{2}I_{4}}^{1/2}\langle 2 \rangle $ \\ \hline
			3& $C_3$ branch& $ \langle 3 - 4 \rangle $& $ \Delta_{\mathbb{R}I_{2}}^{1/2}\langle 3 \rangle $ \\ \hline
			4& Node between $C_1$, $C_2$ and $C_3$& $ \langle 4 - 5 - 6 \rangle $& $ \Delta \langle 4,5,6 \rangle $ \\ \hline
			5& Tangency between $C_3$ and $C_2$& $ \langle 2 - 3 \rangle $& $ \Delta^{2}\langle 2, 3 \rangle $ \\ \hline
			6& Node between $L_1$ and $C_3$& $ \langle 6 - 7 \rangle $& $ \Delta\langle 6, 7 \rangle $ \\ \hline
			7& Tangency between $C_2$ and $L_1$& $ \langle 5 - 6 \rangle $& $ \Delta^{2}\langle 5, 6 \rangle $ \\ \hline
			8& Node between $C_2$ and $C_1$& $ \langle 4 - 5 \rangle $& $ \Delta\langle 4, 5 \rangle $ \\ \hline
			9& Tangency between $L_1$ and $C_1$& $ \langle 5 - 6 \rangle $& $ \Delta^{2}\langle 5, 6 \rangle $ \\ \hline
			10& Node between $C_3$ and $L_1$& $ \langle 6 - 7 \rangle $& $ \Delta\langle 6, 7 \rangle $ \\ \hline
			11& Node between $C_3$, $C_1$ and $C_2$& $ \langle 4 - 5 - 6 \rangle $& $ \Delta \langle 4,5,6 \rangle $ \\ \hline
			12& Tangency between $C_3$ and $C_1$& $ \langle 4 - 5 \rangle $& $ \Delta^{2}\langle 4, 5 \rangle $ \\ \hline
			13& $C_3$ branch& $ \langle 3 - 4 \rangle $& $ \Delta_{I_{2}\mathbb{R}}^{1/2}\langle 3 \rangle $ \\ \hline
			14& Node between $C_1$ and $C_2$& $ \langle 2 - 3 \rangle $& $ \Delta\langle 2, 2 \rangle $ \\ \hline
			15& $C_1$ branch& $ \langle 1 - 2 \rangle $& $ \Delta_{I_{4}I_{2}}^{1/2}\langle 1 \rangle $ \\ \hline
			16& $C_2$ branch& $ \langle 1 - 2 \rangle $& $ \Delta_{I_{6}I_{4}}^{1/2}\langle -1 \rangle $
			\\ \hline \hline
		\end{tabular}
	\end{center}
	\caption{Monodromy table of $\mcC_1$}\label{table:monodromy_table_C1_example}
\end{table}

\begin{figure}[H]
	\begin{center}
		\begin{tikzpicture}
			\begin{scope}[scale=0.7 , xshift=-200]
				\draw [] (5.000000, 0.000000) -- (6.000000, 0.000000) ;
				\node [draw, circle, color=black, fill=white] (vert_0) at (0.000000,0.000000) {$  $};
				\node [draw, circle, color=black, fill=white] (vert_1) at (1.000000,0.000000) {$  $};
				\node [draw, circle, color=black, fill=white] (vert_2) at (2.000000,0.000000) {$  $};
				\node [draw, circle, color=black, fill=white] (vert_3) at (3.000000,0.000000) {$  $};
				\node [draw, circle, color=black, fill=white] (vert_4) at (4.000000,0.000000) {$  $};
				\node [draw, circle, color=black, fill=white] (vert_5) at (5.000000,0.000000) {$  $};
				\node [draw, circle, color=black, fill=white] (vert_6) at (6.000000,0.000000) {$  $};
				
			\end{scope}[]
			\begin{scope}[scale=0.7 , xshift=0]
				\draw [->] (0.000000, 0.000000) -- node [above, midway] {$ \Delta^2\langle 5, 6 \rangle $}  (4.000000, 0.000000) ;
				
			\end{scope}[]

		\begin{scope}[scale=0.7 , xshift=150]
			\draw [] (5.000000, 0.000000) arc (0.000000:-180.000000:0.750000);
			\draw [] (3.500000, 0.000000) arc (180.000000:0.000000:1.250000);
			\node [draw, circle, color=black, fill=white] (vert_0) at (0.000000,0.000000) {$  $};
			\node [draw, circle, color=black, fill=white] (vert_1) at (1.000000,0.000000) {$  $};
			\node [draw, circle, color=black, fill=white] (vert_2) at (2.000000,0.000000) {$  $};
			\node [draw, circle, color=black, fill=white] (vert_3) at (3.000000,0.000000) {$  $};
			\node [draw, circle, color=black, fill=white] (vert_4) at (4.000000,0.000000) {$  $};
			\node [draw, circle, color=black, fill=white] (vert_5) at (5.000000,0.000000) {$  $};
			\node [draw, circle, color=black, fill=white] (vert_6) at (6.000000,0.000000) {$  $};
			\end{scope}
		\end{tikzpicture}
	\end{center}
\caption{Example of a skeleton and an action of a diffeomorphism on it.}\label{fig:skeleton_diffeomorphism_example}
\end{figure}
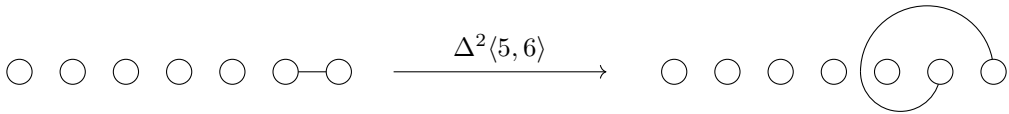

A monodromy table as in Table \ref{table:monodromy_table_C1_example} consists of a row for every singularity (including branch points of the projection to the $x$-axis), ordered by decreasing $x$-coordinate.
In each such row we indicate the singularity number, its description, the associated skeleton as described above, and an associated diffeomorphism that acts on the braids as we pass from a fiber to the left of the singularity to the fiber (of the vertical projection) to the right of the singularity.
Each  diffeomorphism can be one of the following possibilities:

\begin{itemize}
\item For a branch point, the diffeomorphism is $ \Delta^{1/2}\langle i \rangle $ and it changes the points $ i $ and $ i+1 $ from real to imaginary, or vice versa, where the exact action is indicated by a subscript.
\item For a node, the diffeomorphism is $ \Delta\langle i,i+1 \rangle $ and it is a counter-clockwise half-twist of the points $ i $  and $i+1$.
\item For a tangency, the diffeomorphism is $ \Delta^2\langle i,i+1 \rangle $ and it is a counter-clockwise full-twist of the points $ i $  and $i+1$. As an example we can look at the right side of Figure \ref{fig:skeleton_diffeomorphism_example}, in which a diffeomorphism $ \Delta^2\langle 5, 6 \rangle $ is acting on the skeleton $ \langle 6,7 \rangle $.
\item For an intersection point of three components, the diffeomorphism is $ \Delta\langle  i, i+1, i+2 \rangle $, and it is a counter-clockwise half-twist of the points $i, i+1$, and $i+2$.
\end{itemize}

To get the appropriate braid for a certain singularity, we will take its skeleton and act on it with all the diffeomorphisms that correspond to the singularities before it, in reverse order, one after the other, until the diffeomorphism of the first point is activated.
This will produce a braid in the rightmost fiber.
This braid should be understood as describing the action of $[\gamma]\in \pi_1(\mathbb{C}^1 \setminus \{q_1, \dots, q_N\} )$ on $\pi_1(\mathbb{C}^2 \setminus \mathcal{C}_i)$ by moving the endpoints of the braid along it in a way that depends on the singularity type.
The resulting relation is:
\begin{enumerate}
	\item For a branch point in a conic (say conic $C_1$), the relation is $\alpha = \alpha'$.
	\item For a node of a line and a conic (say $L_1$ and $C_3$), the relation is  $[\delta_1,\gamma]=\delta_1\gamma\delta_1^{-1}\gamma^{-1}=e$.
	\item For a tangency point between a line and a conic (say $L_1$ and $C_1$), the relation is \\ $\{\delta_1,\alpha\}=\delta_1\alpha\delta_1\alpha\delta_1^{-1}\alpha^{-1}\delta_1^{-1}\alpha^{-1}=e$.
	\item For an intersection of three components that belong to $C_1, C_2, C_3$, the relation is $\alpha\beta\gamma=\gamma\alpha\beta=\beta\gamma\alpha$.
	\item The projective relation is a product of all generators in the group (i.e., $\alpha,\alpha',\beta,\beta',\gamma,\gamma',\delta_1$), according to the order they appear on the typical fiber.
\end{enumerate}

The generators of the fundamental group will be taken from the following notation:
\begin{notation}\label{not}
Given the conic-line arrangements ${\mcC}_i$ ($i=1, 3$), we construct the generators of the fundamental groups $\pi_1(\mathbb{P}^2 \setminus {\mcC}_i)$ as follows: $\alpha$ and $\alpha'$ are the two loops coming from a general point on the typical fiber, circling the two components of the conic $C_1$ and returning back to the general point. In the same way, we construct $\beta$ and $\beta'$ that correspond to the conic $C_2$ and the generators $\gamma$ and $\gamma'$ that correspond to the conic $C_3$. The  generator $\delta_i$ corresponds to line $L_i$ in ${\mcC}_i$ ($i=1, 3$).
\end{notation}

In the following subsections we compute and determine two fundamental groups $\pi_1(\mathbb{P}^2 \setminus {\mcC}_i)$ ($i=1, 3$). Due to the reason that we apply the diffeomorphisms on the braids,  the relations in the groups appear with conjugations of the generators of the groups.

\subsection{Curve \texorpdfstring{${\mcC}_1$}{C1} and related fundamental group}\label{sec:C1}

\begin{prop}
For the curve ${\mcC}_1$ as defined in the introduction (see Figure \ref{curveC1}),
the fundamental group $\pi_1(\mathbb{P}^2 \setminus {\mcC}_1)$ is free abelian with three generators.

\begin{figure}[H]
\begin{center}
\begin{scaletikzpicturetowidth}{\textwidth}
\begin{tikzpicture}
\begin{scope}[scale = 0.25]
	\draw [rotate around={90.000000:(-0.000000,-0.000000)}, line width=1pt] (-0.000000,-0.000000) ellipse (10.000000 and 10.000000);
	\node  (C_2_naming_node) at (5.000000,10.000000) {$ C_2 $};
	\draw [rotate around={0.000000:(-0.000000,-1.000000)}, line width=1pt] (-0.000000,-1.000000) ellipse (8.000000 and 11.000000);
	\node  (C_3_naming_node) at (7.000000,-1.000000) {$ C_3 $};
	\draw [rotate around={160.238450:(3.959925,8.085314)}, line width=1pt] (3.959925,8.085314) ellipse (10.624951 and 20.112750);
	\node  (C_1_naming_node) at (-7.509300,11.677678) {$ C_1 $};
	
	\draw [line width=1pt] (-10.5, -3) -- (-0.5, -15) node [near end, below] {$L_1$};
	

\end{scope}
\end{tikzpicture}
\end{scaletikzpicturetowidth}
\caption{The curve $\mcC_1$}\label{curveC1}
\end{center}
\end{figure}
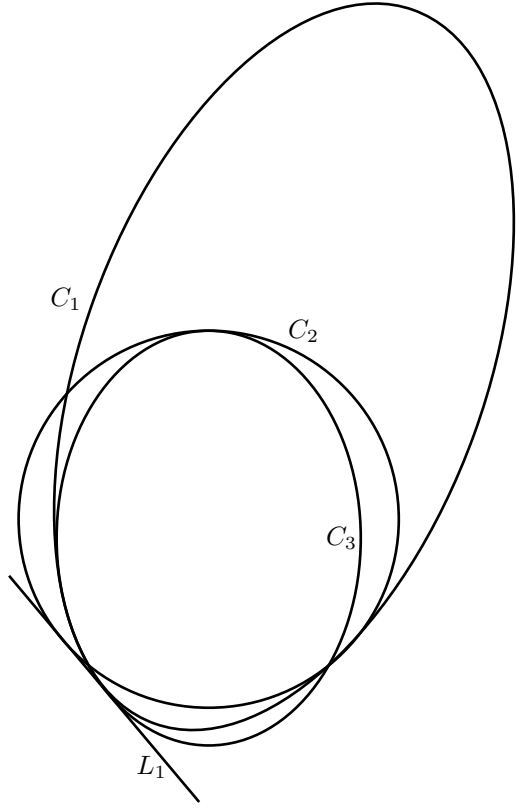

\end{prop}

\begin{proof}
The group $\pi_1(\mathbb{P}^2 \setminus {\mcC}_1)$ has seven generators, which are given in Notation \ref{not}. We recall them here as follows: generators $\alpha, \alpha'$ correspond to conic $C_1$; generators $\beta, \beta'$ correspond to conic $C_2$; generators $\gamma, \gamma'$ correspond to conic $C_3$; and  generator $\delta_1$ corresponds to line $L_1$. The braid monodromy algorithm provides us the braids related to the singularities in ${\mcC}_1$. Then, by the Van Kampen theorem on those braids, we get a presentation for the group with the above generators and the following list of relations:
\begin{equation}\label{7-1}
\begin{split}
\alpha  = \alpha',
\end{split}
\end{equation}
\begin{equation}\label{7-2}
\begin{split}
\alpha'\beta\alpha'^{-1}  = \beta',
\end{split}
\end{equation}
\begin{equation}\label{7-3}
\begin{split}
\beta'\alpha'\gamma\alpha'^{-1}\beta'^{-1}  = \gamma',
\end{split}
\end{equation}
\begin{equation}\label{7-4}
\begin{split}
\left[\alpha', \beta'\gamma'  \right]=e,
\end{split}
\end{equation}
\begin{equation}\label{7-5}
\begin{split}
\left[\beta', \gamma'\alpha'  \right]=e,
\end{split}
\end{equation}
\begin{equation}\label{7-6}
\begin{split}
\left\{\beta , \gamma  \right\}=e,
\end{split}
\end{equation}
\begin{equation}\label{7-7}
\begin{split}
\left[\gamma', \delta_1  \right]=e,
\end{split}
\end{equation}
\begin{equation}\label{7-8}
\begin{split}
\left\{\beta', \delta_1  \right\}=e,
\end{split}
\end{equation}
\begin{equation}\label{7-9}
\begin{split}
\left[\delta_1^{-1}\alpha'\delta_1 , \beta'  \right]=e,
\end{split}
\end{equation}
\begin{equation}\label{7-10}
\begin{split}
\left\{\alpha', \delta_1  \right\}=e,
\end{split}
\end{equation}
\begin{equation}\label{7-11}
\begin{split}
\left[\beta'\alpha'\delta_1\alpha'^{-1}\beta'^{-1} , \gamma'  \right]=e,
\end{split}
\end{equation}
\begin{equation}\label{7-12}
\begin{split}
\left[\beta'^{-1}\gamma'\beta', \alpha'\beta'  \right]=e,
\end{split}
\end{equation}
\begin{equation}\label{7-13}
\begin{split}
\left[\alpha', \gamma'\beta'  \right]=e,
\end{split}
\end{equation}
\begin{equation}\label{7-14}
\begin{split}
\left\{\beta'^{-1}\gamma'\beta', \alpha'  \right\}=e,
\end{split}
\end{equation}
\begin{equation}\label{7-15}
\begin{split}
\beta'\alpha'^{-1}\beta'^{-1}\gamma'^{-1}\delta_1^{-1}\beta\gamma\beta^{-1}\delta_1\gamma'\beta'\alpha'\beta'^{-1}  = \gamma',
\end{split}
\end{equation}
\begin{equation}\label{7-16}
\begin{split}
\left[\beta'^{-1}\gamma'^{-1}\beta'\alpha'^{-1}\beta'^{-1}\gamma'^{-1}\delta_1^{-1}\beta\gamma\beta\gamma^{-1}\beta^{-1}\delta_1\gamma'\beta'\alpha'\beta'^{-1}\gamma'\beta', \alpha'  \right]=e,
\end{split}
\end{equation}
\begin{equation}\label{7-17}
\begin{split}
\beta'^{-1}\gamma'^{-1}\beta'\alpha'^{-1}\beta'^{-1}\gamma'^{-1}\delta_1^{-1}\beta\gamma\beta^{-1}\gamma^{-1}\beta^{-1}\alpha\beta\gamma\beta\gamma^{-1}\beta^{-1}\delta_1\gamma'\beta'\alpha'\beta'^{-1}\gamma'\beta'  = \alpha',
\end{split}
\end{equation}
\begin{equation}\label{7-18}
\begin{split}
\gamma'\beta'\alpha'\beta'^{-1}\gamma'^{-1}\beta'\alpha'^{-1}\beta'^{-1}\gamma'^{-1}\delta_1^{-1}\beta\gamma\beta\gamma^{-1}\beta^{-1}\delta_1\gamma'\beta'\alpha'\beta'^{-1}\gamma'\beta'\alpha'^{-1}\beta'^{-1}\gamma'^{-1}  = \beta',
\end{split}
\end{equation}
\begin{equation}\label{7-19}
\begin{split}
\delta_1\gamma'\beta'\alpha'\alpha\beta\gamma =e.
\end{split}
\end{equation}

Because we have \eqref{7-1} and \eqref{7-2}, we simplify \eqref{7-3} to be $\gamma'=\alpha\beta\gamma\beta^{-1} \alpha^{-1}$.
We substitute this expression together with \eqref{7-1} and \eqref{7-2} in \eqref{7-12} and get $\left[\gamma, \alpha\beta \right]=e$. Therefore, $\gamma'=\alpha\beta\gamma\beta^{-1} \alpha^{-1}$ becomes $\gamma'=\gamma$. Now the substitutions are much easier. For example, using the above simplifications in \eqref{7-13} gives us $\left[\alpha, \beta\gamma \right]=e$.
By both $\left[\gamma, \alpha\beta \right]=e$ and $\left[\alpha, \beta\gamma \right]=e$, we deduce that $\left[\beta, \gamma\alpha \right]=e$.

The above simplified relations simplify the presentation further, and \eqref{7-11} and \eqref{7-16} are now redundant. We have a simplified presentation:

\begin{equation}\label{7-2a}
\begin{split}
\alpha\beta\alpha^{-1}  = \beta',
\end{split}
\end{equation}
\begin{equation}\label{7-4a}
\begin{split}
\left[\alpha ,\alpha\beta\alpha^{-1}\gamma  \right]=e,
\end{split}
\end{equation}
\begin{equation}\label{7-5a}
\begin{split}
\left[\alpha\beta\alpha^{-1} , \gamma\alpha  \right]=e,
\end{split}
\end{equation}
\begin{equation}\label{7-6a}
\begin{split}
\left\{\beta , \gamma  \right\}=e,
\end{split}
\end{equation}
\begin{equation}\label{7-7a}
\begin{split}
\left[\gamma , \delta_1  \right]=e,
\end{split}
\end{equation}
\begin{equation}\label{7-8a}
\begin{split}
\left\{\alpha\beta\alpha^{-1} , \delta_1  \right\}=e,
\end{split}
\end{equation}
\begin{equation}\label{7-9a}
\begin{split}
\left[\delta_1\alpha\delta_1^{-1} , \beta  \right]=e,
\end{split}
\end{equation}
\begin{equation}\label{7-10a}
\begin{split}
\left\{\alpha , \delta_1  \right\}=e,
\end{split}
\end{equation}
\begin{equation}\label{7-12a}
\begin{split}
\left[\gamma, \alpha\beta  \right]=e,
\end{split}
\end{equation}
\begin{equation}\label{7-13a}
\begin{split}
\left[\alpha , \beta\gamma  \right]=e,
\end{split}
\end{equation}
\begin{equation}\label{7-13*a}
\begin{split}
\left[\beta , \gamma\alpha  \right]=e,
\end{split}
\end{equation}
\begin{equation}\label{7-14a}
\begin{split}
\left\{\alpha , \gamma  \right\}=e,
\end{split}
\end{equation}
\begin{equation}\label{7-15a}
\begin{split}
\beta^{-1}\alpha^{-1}\delta_1^{-1}\beta\gamma\beta^{-1}
\delta_1\alpha\beta  = \alpha^{-1}\gamma \alpha,
\end{split}
\end{equation}
\begin{equation}\label{7-17a}
\begin{split}
\gamma\alpha\gamma^{-1}=\delta_1\alpha\delta_1^{-1} ,
\end{split}
\end{equation}
\begin{equation}\label{7-18a}
\begin{split}
\gamma^{-1}\beta\gamma=\delta_1^{-1}\beta\delta_1 ,
\end{split}
\end{equation}
\begin{equation}\label{7-19a}
\begin{split}
\delta_1(\alpha\beta\gamma)^2=e .
\end{split}
\end{equation}

We use \eqref{7-19a} to eliminate generator $\delta_1$. In this elimination process, the relations  \eqref{7-4a}, \eqref{7-5a},\eqref{7-7a}, \eqref{7-10a}, and \eqref{7-15a}, become redundant.

Now, relation \eqref{7-9a} simplifies to $\left[\alpha , \beta \right]=e$, and this relation simplifies \eqref{7-2a} to $\beta'=\beta$. Moreover, relation \eqref{7-8a} becomes $\left[\beta , \gamma \right]=e$. We conclude easily that $\left[\alpha , \gamma \right]=e$ as well. All these commutations make \eqref{7-17a} and \eqref{7-18a} redundant.

Therefore, group $\pi_1(\mathbb{P}^2 \setminus {\mcC}_1)$ is free abelian and is generated by  generators $\alpha, \beta, \gamma$.

\end{proof}

\subsection{Curve \texorpdfstring{${\mcC}_3$}{C3} and related fundamental group}\label{sec:C3}

\begin{prop}
	For the curve ${\mcC}_3$ as defined in the introduction (see Figure \ref{curveC3}),
	the fundamental group $\pi_1(\mathbb{P}^2 \setminus {\mcC}_3)$ is free abelian with three generators.

\begin{figure}[H]
	\begin{center}
		\begin{scaletikzpicturetowidth}{\textwidth}
			\begin{tikzpicture}
				\begin{scope}[scale = 0.25]
					\draw [rotate around={90.000000:(-0.000000,-0.000000)}, line width=1pt] (-0.000000,-0.000000) ellipse (10.000000 and 10.000000);
					\node  (C_2_naming_node) at (5.000000,10.000000) {$ C_2 $};
					\draw [rotate around={0.000000:(-0.000000,-1.000000)}, line width=1pt] (-0.000000,-1.000000) ellipse (8.000000 and 11.000000);
					\node  (C_3_naming_node) at (7.000000,-1.000000) {$ C_3 $};
					\draw [rotate around={160.238450:(3.959925,8.085314)}, line width=1pt] (3.959925,8.085314) ellipse (10.624951 and 20.112750);
					\node  (C_1_naming_node) at (5,25) {$ C_1 $};
					
					\draw [line width=1pt] (-11, -0.5) -- (-1, 24.2) node [near end, left] {$L_3$};
					

				\end{scope}
			\end{tikzpicture}
		\end{scaletikzpicturetowidth}
		\caption{The curve $\mcC_3$}\label{curveC3}
	\end{center}
\end{figure}

\end{prop}

\begin{proof}
The group $\pi_1(\mathbb{P}^2 \setminus {\mcC}_3)$ has seven generators, which are $\alpha, \alpha'$ (related to conic $C_1$),  $\beta, \beta'$ (related to $C_2$),  $\gamma, \gamma'$ (related to $C_3$),   and $\delta_3$ (related to line $L_3$). The group admits the following presentation:
\begin{equation}\label{7C2-1}
\begin{split}
\alpha  = \alpha',
\end{split}
\end{equation}
\begin{equation}\label{7C2-2}
\begin{split}
\alpha'\beta\alpha'^{-1}  = \beta',
\end{split}
\end{equation}
\begin{equation}\label{7C2-3}
\begin{split}
\beta'\alpha'\gamma\alpha'^{-1}\beta'^{-1}  = \gamma',
\end{split}
\end{equation}
\begin{equation}\label{7C2-4}
\begin{split}
\left[\alpha', \beta'\gamma'  \right]=e,
\end{split}
\end{equation}
\begin{equation}\label{7C2-5}
\begin{split}
\left[\beta', \gamma'\alpha'  \right]=e,
\end{split}
\end{equation}
\begin{equation}\label{7C2-6}
\begin{split}
\left\{\beta , \gamma  \right\}=e,
\end{split}
\end{equation}
\begin{equation}\label{7C2-7}
\begin{split}
\left[\alpha', \beta'  \right]=e,
\end{split}
\end{equation}
\begin{equation}\label{7C2-8}
\begin{split}
\left[\gamma', \delta_3  \right]=e,
\end{split}
\end{equation}
\begin{equation}\label{7C2-9}
\begin{split}
\left\{\beta'\alpha'\beta'^{-1} , \delta_3  \right\}=e,
\end{split}
\end{equation}
\begin{equation}\label{7C2-10}
\begin{split}
\left[\beta'\alpha'\beta'^{-1}\delta_3\beta'\alpha'^{-1}\beta'^{-1} , \gamma'  \right]=e,
\end{split}
\end{equation}
\begin{equation}\label{7C2-11}
\begin{split}
\left[\beta'^{-1}\gamma'\beta', \alpha'\beta'^{-1}\delta_3^{-1}\beta'\delta_3\beta'  \right]=e,
\end{split}
\end{equation}
\begin{equation}\label{7C2-12}
\begin{split}
\left[\alpha', \beta'^{-1}\delta_3^{-1}\beta'\delta_3\gamma'\beta'  \right]=e,
\end{split}
\end{equation}
\begin{equation}\label{7C2-13}
\begin{split}
\left\{\beta', \delta_3  \right\}=e,
\end{split}
\end{equation}
\begin{equation}\label{7C2-14}
\begin{split}
\left\{\beta'^{-1}\gamma'\beta', \alpha'  \right\}=e,
\end{split}
\end{equation}
\begin{equation}\label{7C2-15}
\begin{split}
\beta'\alpha'^{-1}\beta'^{-1}\gamma'^{-1}\delta_3^{-1}\beta\gamma\beta^{-1}\delta_3\gamma'\beta'\alpha'
\beta'^{-1}  = \gamma',
\end{split}
\end{equation}
\begin{equation}\label{7C2-16}
\begin{split}
\left[\beta'^{-1}\gamma'^{-1}\beta'\alpha'^{-1}\beta'^{-1}\gamma'^{-1}\delta_3^{-1}\beta\gamma\beta
\gamma^{-1}\beta^{-1}\delta_3\gamma'\beta'\alpha'\beta'^{-1}\gamma'\beta', \alpha'  \right]=e,
\end{split}
\end{equation}
\begin{equation}\label{7C2-17}
\begin{split}
\beta'^{-1}\gamma'^{-1}\beta'\alpha'^{-1}\beta'^{-1}\gamma'^{-1}\delta_3^{-1}\beta\gamma\beta^{-1}
\gamma^{-1}\beta^{-1}\alpha\beta\gamma\beta\gamma^{-1}\beta^{-1}\delta_3\gamma'\beta'\alpha'\beta'^{-1}
\gamma'\beta'  = \alpha',
\end{split}
\end{equation}
\begin{equation}\label{7C2-18}
\begin{split}
\gamma'\beta'\alpha'\beta'^{-1}\gamma'^{-1}\beta'\alpha'^{-1}\beta'^{-1}\gamma'^{-1}\delta_3^{-1}
\beta\gamma\beta\gamma^{-1}\beta^{-1}\delta_3\gamma'\beta'\alpha'\beta'^{-1}\gamma'\beta'\alpha'^{-1}
\beta'^{-1}\gamma'^{-1}  = \beta',
\end{split}
\end{equation}
\begin{equation}\label{7C2-19}
\begin{split}
\delta_3\gamma'\beta'\alpha'\alpha\beta\gamma =e.
\end{split}
\end{equation}

As \eqref{7C2-7} is $\left[\alpha , \beta  \right]=e$, we get in \eqref{7C2-2} that $\beta'=\beta$. The result $\beta'=\beta$, along with $\alpha'=\alpha$ (from \eqref{7C2-1}) and $\left[\alpha , \beta  \right]=e$, simplify \eqref{7C2-4} and  \eqref{7C2-5} to $\left[\alpha , \gamma  \right]=e$ and $\left[\beta , \gamma  \right]=e$ respectively. By these resulting commutations, \eqref{7C2-3} becomes $\gamma'=\gamma$. It is much easier now to simplify the presentation, having $\alpha'=\alpha$, $\beta'=\beta$, and $\gamma'=\gamma$. The group  now has generators $\alpha$, $\beta$, $\gamma$, and $\delta_3$ and admits the following relations:
\begin{equation}\label{7C2-1a}
\begin{split}
\left[\alpha , \beta  \right]=e,
\end{split}
\end{equation}
\begin{equation}\label{7C2-2a}
\begin{split}
\left[\alpha , \gamma  \right]=e,
\end{split}
\end{equation}
\begin{equation}\label{7C2-3a}
\begin{split}
\left[\beta , \gamma  \right]=e,
\end{split}
\end{equation}
\begin{equation}\label{7C2-4a}
\begin{split}
\left[\alpha , \delta_3  \right]=e,
\end{split}
\end{equation}
\begin{equation}\label{7C2-5a}
\begin{split}
\left[\beta , \delta_3  \right]=e,
\end{split}
\end{equation}
\begin{equation}\label{7C2-6a}
\begin{split}
\left[\gamma , \delta_3  \right]=e,
\end{split}
\end{equation}
\begin{equation}\label{7C2-7a}
\begin{split}
\delta_3(\gamma\beta\alpha)^2=e.
\end{split}
\end{equation}

We use \eqref{7C2-7a} to write $\delta_3 = (\gamma\beta\alpha)^{-2}$, and then we substitute it in \eqref{7C2-4a}, \eqref{7C2-5a}, and \eqref{7C2-6a}; these relations are redundant.
Therefore, group $\pi_1(\mathbb{P}^2 \setminus {\mcC}_3)$ is free abelian with generators $\alpha, \beta, \gamma$.
\end{proof}


\bibliographystyle{spmpsci}
\bibliography{references-degree7}

\end{document}